\documentclass[12pt,reqno]{amsart}
\usepackage{graphicx}
\usepackage[utf8]{inputenc}  % allow utf-8 input
\usepackage{hyperref}        % hyperlinks
\usepackage{url}             % simple URL typesetting
\usepackage{booktabs}        % professional-quality tables
\usepackage{amsfonts}        % blackboard math symbols
\usepackage{amsmath,amssymb}
\usepackage{nicefrac}        % compact symbols for 1/2, etc.
\usepackage{microtype}       % microtypography
\usepackage{lmodern}
\usepackage{fullpage}
\usepackage{amssymb,floatflt,amscd,amsthm,graphics}
\usepackage{subcaption}
\usepackage{graphicx}     % microtypography
\newtheorem{lemma}{Lemma}[section]
\newtheorem{proposition}{Proposition}[section]
\newtheorem{theorem}{Theorem}[section]

\newtheorem{definition}{Definition}[section]
\newtheorem{remark}{\textbf{Remark}}[section]

\newtheorem{example}{\textbf{Example}}[section]
\newtheorem{assumption}{\textbf{Assumption}}[section]
\usepackage{subcaption}
\newtheorem{algorithm}{Algorithm}[section]
\newcommand{\iprod}[2]{\left\langle {#1}, {#2} \right\rangle}
\newcommand{\prox}{\mathrm{prox}}
\newcommand{\Espace}{\mathcal{E}}

\newcommand{\argmin}{\arg\min}

\newcommand{\norm}[1]{\left\lVert {#1} \right\rVert}
\newcommand{\bracket}[1]{\left( {#1} \right)}

\newcommand{\grad}{\nabla}

\newcommand{\domain}{\mathrm{dom}}

\newcommand{\expect}{\mathbb{E}}

\newcommand{\lea}{\stackrel{\rm(a)}{\le}}
\newcommand{\leb}{\stackrel{\rm(b)}{\le}}
\newcommand{\lec}{\stackrel{\rm(c)}{\le}}

\usepackage{algorithm}
\usepackage[noend]{algpseudocode}

\makeatletter
\renewcommand\fs@ruled{%
  \def\@fs@cfont{\rmfamily}%
  \let\@fs@capt\floatc@plain%
  \def\@fs@pre{\hrule height.8pt depth0pt \kern2pt}%% or use \def\@fs@pre{} to get rid of the top rule
  \def\@fs@post{\kern2pt\hrule\relax}%% or use \def\@fs@post{} to get rid of the last rule
  \def\@fs@mid{\kern2pt\hrule\kern2pt}%% or use \def\@fs@mid{} to get rid of the middle rule
  \let\@fs@iftopcapt\iffalse}
\makeatother

\begin{document}
 \title{Accelerated Randomized Mirror Descent Algorithms For Composite Non-strongly Convex Optimization}
\author{Le Thi Khanh Hien}
\address[Le Thi Khanh Hien]{Department of Industrial and Systems Engineering and Management\\
  National University of Singapore \\
  Singapore 117576}
\email{iseltkh@nus.edu.sg}
\author{Cuong V. Nguyen}
\address[Cuong V. Nguyen]{Department of Engineering \\ University of Cambridge}
\email{vcn22@cam.ac.uk}
\author{Huan Xu}
\address[Huan Xu]{School of Industrial \& Systems Engineering\\
     Georgia Institute of Technology}
\email{huan.xu@isye.gatech.edu}
\author{Canyi Lu}
\address[Canyi Lu]{ Department of Electrical and Computer Engineering  \\
     National University of Singapore \\
     Singapore 117583}
\email{canyilu@u.nus.edu.sg}
\author{Jiashi Feng}
\address[Jiashi Feng]{Department of Electrical and Computer Engineering  \\
     National University of Singapore \\
     Singapore 117583}
\email{elefjia@nus.edu.sg }

\maketitle
\begin{abstract}
We consider the problem of minimizing the sum of an average function of a large number of smooth convex components and a general, possibly non-differentiable, convex function. Although many methods have been proposed to solve this problem with the assumption that the sum is strongly convex, few methods support the non-strongly convex case. Adding a small quadratic regularization is a common trick used to tackle non-strongly convex problems; however, it may cause loss of sparsity of solutions or weaken the performance of the algorithms. Avoiding this trick, we propose an accelerated randomized mirror descent method for solving this problem without the strongly convex assumption. Our method extends the deterministic accelerated proximal gradient methods of Paul Tseng and can be applied even when proximal points are computed inexactly. We also propose a scheme for solving the problem when the component functions are non-smooth.
\end{abstract}

%\KEYWORDS{}
%\MSCCLASS{}
%\ORMSCLASS{}


\maketitle

\section{Introduction} 
We let $\mathcal{E}$ be a finite dimensional real linear space endowed with a norm $\norm{\cdot}$ and let $\mathcal{E}^*$ be the space of continuous linear functionals on $\mathcal{E}$. We use $\iprod{x^*}{x}$ to denote the value of $x^*\in\mathcal{E}^*$ at $x\in \mathcal{E}$, and $\norm{\cdot}_*$ to denote the dual norm, i.e., $\norm{x^*}_*=\sup_{\norm{x}\leq 1}\iprod{x^*}{x}$. We consider the following composite convex optimization problem:
\begin{equation}\label{eq:composite}
\min\limits_{x\in \mathcal{E}} \left\{ F^P(x):=F(x) + P(x)\right\},
\end{equation}
where $F(x)=\frac{1}{n}\sum\limits_{i=1}^n f_i(x)$. Throughout this paper we focus on problems satisfying the following assumption.
\begin{assumption}
\label{asstn:convexsmooth}
Function $P(x)$ is lower semi-continuous and convex. The domain of $P(x)$, $\domain(P)=\{x\in \Espace, P(X)< +\infty\}$, is closed. Each function $f_i(x)$ is convex and $L_i$-Lipschitz smooth, i.e., it is differentiable on an open set containing $\domain(P)$ and its gradient is Lipschitz continuous with constant $L_i$:
\[
\norm{\grad f_i(x)-\grad f_i(y)}_*\leq L_i \norm{x-y}, \forall x,y \in \domain(P).
\]
\end{assumption}
Problems of this form often appear in machine learning and statistics. For examples, in $l_1$-regularized logistic regression, we have $f_i(x)= \log (1+ e^{-b_i\iprod{a_i}{x}})$ where $a_i\in \mathbb{R}^p$, $b_i\in \{-1,1\}$, and $P(x)= \lambda \norm{x}_1$. In Lasso, we have $f_i(x)= \frac12 (\iprod{a_i}{x}-b_i)^2$ and $P(x)=\lambda \norm{x}_1$. More generally, any $l_1$-regularized empirical risk minimization problem 
$\min\limits_{x\in \mathbb{R}^p} \left\{\frac{1}{n} \sum\limits_{i=1}^n f_i(x) + \lambda \norm{x}_1\right\}$
with smooth convex loss functions $f_i$ belongs to the framework \eqref{eq:composite}. We can also use the function $P(x)$ for modelling purpose. For example, when $P(x)$ is the indicator function, i.e., $P(x)= 0$ if $x\in X$ and $P(x)= \infty$ otherwise, \eqref{eq:composite} becomes the popular constrained finite sum optimization problem
$
\min\limits_{x\in X} \left\{F(x):=\frac{1}{n}\sum\limits_{i=1}^n f_i(x)\right\}.
$

One well-known method to solve \eqref{eq:composite} is the proximal gradient descent (PGD) method. Let the proximal mapping of a convex function $P(x)$ be defined as:
$$\prox_P(x)=\argmin\limits_u \left\{P(u)+\frac12\norm{u-x}_2^2\right\}.$$ 
At each iteration, PGD calculates a proximal point:
\[
\begin{split}
x_{k}=\prox_{\gamma_k P}(x_{k-1}-\gamma_k \grad F(x_{k-1})) =\argmin\limits_{x} \Big\{ \iprod{\grad F(x_{k-1})}{x} + \frac{1}{2\gamma_k}\norm{x-x_{k-1}}^2 + P(x)\Big\},
\end{split}
\]
where $\gamma_k$ is the step size at the $k$-th iteration.
Methods such as gradient descent, which computes $x_k=x_{k-1}-\gamma_k \grad F(x_{k-1})$, or projection gradient descent, which computes $x_k=\Pi_X(x_{k-1}-\gamma_k \grad F(x_{k-1}))$, are in the class of PGD algorithms. 
Indeed, PGD becomes gradient descent when $P(x)=0$ and becomes projection gradient descent when $P(x)$ is the indicator function. If $F(x)$ and $P(x)$ are general convex functions and $F(x)$ is $L$-Lipschitz smooth, then PGD has the convergence rate $O\bracket{{1}/{k}}$.
However, this convergence rate is not optimal.
Nesterov, for the first time in \cite{Nesterov1983}, proposed an acceleration method for solving \eqref{eq:composite} with $P(x)$ being an indicator function and achieved the optimal convergence rate $ O\bracket{{1}/{k^2}}$.
Later in \cite{Nesterov1998,Nesterov,Nesterov2013}, he introduced two other acceleration techniques, which make one or two proximal calls together with interpolation at each iteration to accelerate the convergence. Nesterov's ideas have been further studied and applied for solving many practical optimization problems such as rank reduction in multivariate linear regression and sparse covariance selection (see \cite{Becker2011,Aspremont2008} and reference therein). Auslender and Teboulle \cite{Auslender2006} used the acceleration technique in the context of Bregman divergence $D(x,y)$, which generalizes the squared Euclidean distance $\frac12\norm{x-y}^2$.
Tseng \cite{Tseng} unified the analysis of all these acceleration techniques, proposed new variants and gave simple analysis for the proof of the optimal convergence rate.

When $n$ is very large, applying PGD can be unappealing since computing the full gradient $\grad F(x_{k-1})=\frac{1}{n}\sum_{i=1}^n \grad f_i(x_{k-1})$ in each iteration is very expensive. An effective alternative is the randomized version of PGD which is usually called stochastic proximal gradient descend (SPGD) method:
\[
x_k=\argmin\limits_{x} \left\{ \iprod{\grad f_{i_k}(x_{k-1})}{x} + \frac{1}{2\gamma_k}\norm{x-x_{k-1}}^2 + P(x) \right\},
\]
where $i_k$ is uniformly drawn from $\{1,\ldots,n\}$ at each iteration. For a suitably chosen decreasing step size $\gamma_k$, SPGD was proven to have the suboptimal rate $O({1}/{k})$ in the case of strongly convex $F^P(x)$ \cite{Nemirovski_etal}. Many authors have proposed methods to obtain better convergence rate when $F^P(x)$ is strongly convex; stochastic average gradient (SAG) \cite{Rouxetal}, stochastic variance reduced gradient (SVRG) \cite{Rie}, proximal SVRG \cite{Xiao_Zhang}, and SARAH \cite{Lam2017} are noticeable examples that have a linear convergence rate, which is optimal for this case.

There have been very few algorithms that directly support the non-strongly convex case. One of those is the randomized coordinate gradient descent method (RCGD), which recently has been successfully extended to accelerated versions to achieve the optimal rate $O({1}/{k^2})$ \cite{Fercoq}. However, accelerated RCGD is only applicable to block separable regularization $P(x)$, i.e., $P(x)=\sum P_i(x^{(i)})$ where $x^{(i)}$ and $P_i$ are correspondingly the $i$-th coordinate block of $x$ and $P$. It is worth to mention Catalyst scheme which accelerates first-order methods to achieve a better rate \cite[Alg.~1]{Lin2015}. However, Catalyst approximately solves a sequence of auxiliary problems which are formed by adding a strongly convex regularizer to $F^P(x)$. Therefore, we cannot consider Catalyst as a direct method for solving the non-strongly convex problems. Furthermore, using this scheme for first order method to solve non-strongly convex problems only obtain near-optimal rate \cite[Sect.~3.2]{Lin2015}. To the best of our knowledge, accelerated proximal gradient descent (APG) is the only algorithm that obtains the optimal rate $O({1}/{k^2})$ for directly solving \eqref{eq:composite} under Assumption \ref{asstn:convexsmooth}. Therefore, the main goal of this work is to extend APG to randomized variants that support the non-strongly convex case and even outperform APG on large scale problems.

Our second goal is to solve \eqref{eq:composite} when proximal points with respect to $P(x)$ {\em cannot be computed explicitly}.
For several choices of $P(x)$, the proximal points used in the above-mentioned algorithms can be calculated efficiently, e.g., when $P(x)=\norm{x}_1$, the proximal points are explicitly computed by a soft threshold operator \cite{Boyd}. However, in many cases such as nuclear norm regularization and total variation regularization, it is very expensive to compute them exactly. For that reason, many efficient methods have been proposed to calculate proximal points inexactly \cite{Cai,Fadili,Ma2011}. Basic methods that allow inexact computation of proximal points were first studied by Rockafellar \cite{Rock}. Since then, there has emerged a growing interest in both inexact proximal point and inexact accelerated proximal point algorithms \cite{Cai,Devolder2014,Schmidt2011,Solodov2000,Villa}.
However, although there were many work showing impressive empirical performance of inexact (accelerated) PGD methods, there has been no analysis on their randomized versions. Our work gives such an analysis.

In a concurrent work \cite{AllenZhu}, an \emph{exact} accelerated randomized gradient descent in the setting with Euclidean distance was independently analyzed and the same convergence rate was proven. In comparison, as we extend the general deterministic acceleration framework of Tseng \cite{Tseng} and consider non-uniform sampling together with a broader choice of involving parameters (i.e., $\alpha_{1,s},\alpha_{2,s}$ in \eqref{eq:alphaupdate}), our analysis for \emph{inexact} accelerated randomized mirror descent algorithm is put in a more general framework but employs simpler and neater proofs in the setting with Bregman distance. By contrast, the analysis of \cite{AllenZhu} depends critically on a specific choice of $x_{k,s}$ in Update \eqref{eq:choosexk} of our Algorithm \ref{alg:ASM}. In particular, their proofs critically rely on using Variant 1 of our Example \ref{ex:xksupdate} to update $x_{k,s}$. 

Our main contribution is the incorporation of the variance reduction technique and the general acceleration methods of Tseng to propose a framework of exact as well as inexact accelerated randomized mirror descent (ARMD and inexact ARMD, respectively) algorithms for the non-strongly convex optimization problem \eqref{eq:composite}. At each stage of our inexact algorithms, proximal points are allowed to be calculated inexactly. When the component functions $f_i$ are non-smooth, we give a scheme for minimizing the corresponding non-smooth problem. The rate obtained using our smoothing scheme significantly improves the rate obtained using subgradient methods or stochastic subgradient methods.

\textbf{Paper structure}. Sect~\ref{sec:pre}  provides some preliminaries. The proposed ARMD and inexact ARMD methods and their convergence analysis are presented in Sect.~\ref{sec:iASMD}.  Computational results are given in Sect.~\ref{sec:experiment}, and proofs are in the Appendix.

\section{Preliminaries} \label{sec:pre}
For a given continuous function $f(x)$, a convex set $X$ and a non-negative number $\varepsilon$, we write $z\approx_\varepsilon \argmin\limits_{x\in X} \{f(x)+P(x) \}$ to denote $z\in \domain(P) $ such that $f(z) + P(z)\leq  \min\limits_{x\in X}\{ f(x) + P(x)\} + \varepsilon$. We use $\grad_1 g(x,y)$ to denote the gradient of the function $x\mapsto g(x,y)$ at $x$. 
We now give some important definitions and lemmas that will be used in the paper.
\begin{definition} 
Let $h:\Omega\to \mathbb{R}$ be a strictly convex function that is differentiable on an open set containing $\domain(P)$. The Bregman distance is defined as: \\
${\hskip 1.5cm} D(x,y)=h(x)-h(y)-\iprod{\grad h(y)}{x-y}, \quad \forall y\in \domain(P), x\in \Omega.$
\end{definition}
\begin{example} 
(a) If $h(x)=\frac12\norm{x}_2^2$, then $D(x,y)=\frac12 \norm{x-y}_2^2$ is the Euclidean distance. \\
(b) If $\Omega=\{x\in\mathbb{R}^p_+:\sum_{i=1}^p x_i=1\}$ and $h(x)=\sum_{i=1}^p x_i \log x_i$, then $D(x,y)=\sum_{i=1}^p x_i \log\frac{x_i}{y_i}$ is called the entropy distance.
\end{example}
\begin{lemma}
\label{lem:DgeqE} 
If $h$ is strongly convex with constant $\sigma$, i.e., 
$h(y)\geq h(x) + \iprod{\grad h(x)}{y-x}+\frac{\sigma}{2}\norm{x-y}^2,$
then the corresponding Bregman distance satisfies
$D(x,y)\geq \frac{\sigma}{2}\norm{x-y}^2.$
\end{lemma}
\begin{lemma}
\label{lem:Dinequal}
Let $\phi(x)$ be a proper convex function whose domain is an open set containing $C$. For any $z\in \domain(P)$, if $z^* = \argmin\limits_{x\in C} \{ \phi(x) + D(x,z)\}$, then for all $x\in \domain (P)$ we have:
$\phi(x) + D(x,z) \geq \phi(z^*) + D(z^*,z) + D(x,z^*).$
\end{lemma}
\begin{lemma}
\label{lem:triangle}
Let $D(x,y)$ be the Bregman distance with respect to $h$. We have: \\
${\hskip 1.7cm} D(x,y)+D(y,z)=D(x,z)+\iprod{x-y}{\grad h(z)-\grad h(y)}.$
\end{lemma}
If we replace the Euclidean distance in PGD by Bregman distance, we obtain the mirror descent method:
\[ 
x_{k}=\argmin\limits_{x} \left\{ \iprod{\grad F(x_{k-1})}{x} + P(x)+ \frac{1}{\gamma_k}D(x,x_{k-1}) \right\}. 
\]
We refer the readers to \cite{Bregman1967,Teboulle1997,Tseng} and references therein for proofs of these lemmas and applications of Bregman distance as well as mirror descent methods. Since we can scale $h$ if necessary, we assume $h$ is strongly convex with constant $\sigma=1$ in this paper. The following lemma for $L$-Lipschitz smooth functions (with proof in \cite{Nesterov2004}) is crucial for our analysis.

\begin{lemma}
\label{lem:LiGrad}
If $f_i(\cdot)$ is a convex and $L_i$-Lipschitz smooth function, then: \\
(1) \quad $f_i(x) \leq f_i(y) + \iprod{\grad f_i(y)}{x-y} + \frac{L_i}{2}\norm{x-y}^2$, and \\ 
(2) \quad $\frac{1}{2L_i}\norm{\grad f_i(y)-\grad f_i(x)}_*^2 + \iprod{\grad f_i(y)}{x-y} \leq f_i(x)-f_i(y).$
\end{lemma}

\section{Accelerated Randomized Mirror Descent Methods} 
\label{sec:iASMD}

\subsection{Algorithm Description}

\begin{algorithm}[!t]
\caption{Accelerated Randomized Mirror Descent (ARMD) Algorithm}
\label{alg:ASM}
\begin{algorithmic}[1]
\State Choose $\tilde{x}_0,x_{1,0},\ldots, x_{m,0}, z_{0} \in \domain (P)$, $\alpha_3>0$ and non-negative sequences $\left\{\alpha_{1,s}\right\},\left\{\alpha_{2,s}\right\}$ such that:
\begin{equation} 
\label{eq:alphaupdate}
\begin{cases}
\frac{1-\alpha_{1,s}}{\alpha_{2,s}^2} \geq \frac{\alpha_{3}}{\alpha_{2,s+1}^2}\\
\frac{1}{\alpha_{2,s}^2} \geq \frac{1-\alpha_{2,s+1}}{\alpha_{2,s+1}^2}\\
\alpha_{1,s}+\alpha_{2,s}+\alpha_{3}=1.
\end{cases}
\end{equation}
\State Choose probability $Q=\{q_1,\ldots,q_n\}$ on $\{1,\ldots,n\}$.
\State Denote $L_Q=\max_i \frac{L_i}{q_i n}, L_A=\frac{1}{n}\sum_{i=1}^n L_i$. Let $\overline{L}= L_A+\frac{4L_Q}{\alpha_{3}}$ and $z_{m,0}=z_0$.
\For{$s=1,2,\ldots $}
\State Calculate $\tilde{v}=\nabla F (\tilde{x}_{s-1})$.
\State Let $x_{0,s}=x_{m,s-1}$ and $z_{0,s} = z_{m,s-1}$.
\State Choose a nonempty closed convex set $X_s \subseteq \Espace$ with $X_s \cap \domain (P) \ne \emptyset$.
\For{$k=1,\ldots,m$ }
\State - Pick $i_k\in \{1,\ldots,n \}$ randomly according to $Q$.
\State - Set $y_{k,s}=\alpha_{1,s} x_{k-1,s} + \alpha_{2,s} z_{k-1,s} + \alpha_3 \tilde{x}_{s-1}.$
\State - Set $v_{k}=\tilde{v}+\frac{\nabla f_{i_k} (y_{k,s}) - \nabla f_{i_k} (\tilde{x}_{s-1})}{q_{i_k}n}.$
\State - Let $\theta_s=\alpha_{2,s} \overline{L}$, update:
\begin{equation} 
\label{eq:proxpoint}
z_{k,s} \approx_{\varepsilon_{k,s}} \argmin\limits_{x\in X_s}\left\{ \iprod{v_k}{x} + P(x) + \theta_s  D(x,z_{k-1,s}) \right\}.
\end{equation}
\State - Set $\hat{x}_{k,s} = \alpha_{1,s} x_{k-1,s} + \alpha_{2,s} z_{k,s} + \alpha_3 \tilde{x}_{s-1}$.
\State - Choose $x_{k,s}$ such that:
\begin{equation} \label{eq:choosexk}
\iprod{v_k }{x_{k,s}} + \frac{ \overline{L}}{2}
\norm{x_{k,s}-y_{k,s}}^2 + P(x_{k,s})
\leq \iprod{v_k }{\hat{x}_{k,s}} + \frac{ \overline{L}}{2}
\norm{\hat{x}_{k,s}-y_{k,s}}^2 + P(\hat{x}_{k,s}).
\end{equation}
\EndFor
\State Update $\tilde{x}_s$ such that:
\begin{equation}\label{eq:xsupdate}
F^P(\tilde{x}_s) \leq \frac{1}{m}\sum\limits_{i=1}^m F^P(x_{i,s}).
\end{equation}
\EndFor
\end{algorithmic}
\end{algorithm}

Algorithm \ref{alg:ASM} details our framework. The algorithm has 2 loops - the outer loop indexed by $s$ and the inner loop indexed by $k$. Specifically, we use $x_{k,s}$ to mean that the point is at step $k$ of the inner loop, which belongs to stage $s$ of the outer loop. Before running each inner loop, a full gradient $\nabla F(\cdot)$ is calculated. Each inner loop is executed $m$ steps, i.e., $k=1,\ldots,m$. At step $k$ of an inner loop, under non-uniform sampling setting, we randomly pick one function $f_{i_k}$ to calculate its derivative $\nabla f_{i_k}(\cdot)$, then find the (inexact) proximal point in \eqref{eq:proxpoint} and perform the update \eqref{eq:choosexk}. We let $\varepsilon_{k,s}$ be the error in calculating the proximal points in \eqref{eq:proxpoint}. When $\varepsilon_{k,s}=0$, our algorithm reduces to exact ARMD. The non-uniform sampling method would improve the complexity of our algorithms when $L_i$ are different. The set $X_s$ should be chosen such that it contains a solution of \eqref{eq:composite} (see Proposition \ref{prop:result}). Choosing $X_s=\Espace$ is the simplest variant of $X_s$. To accelerate convergence, we can always choose a smaller $X_s$. We refer the readers to \cite[Sect.~3]{Tseng} for examples of choosing smaller $X_s$ and omit the details here.

We stress that the update rule \eqref{eq:choosexk} of $x_{k,s}$ is very general, and as such we can derive many specific algorithms from the general framework \ref{alg:ASM}. In particular, we give some examples that satisfy \eqref{eq:choosexk} below.

\begin{example}
\label{ex:xksupdate}
\begin{enumerate}
\item $x_{k,s}=x_{k,s}^u$ with $x_{k,s}^u=\hat{x}_{k,s}$. For this choice, each step $k$ of an inner loop only computes one inexact/exact proximal point $z_{k,s}$. 
\item $x_{k,s}=x_{k,s}^l$ with $x_{k,s}^l=\argmin\limits_{x\in X_s}\left\{\iprod{v_k }{x} + \frac{\overline{L}}{2}\norm{x-y_{k,s}}^2 + P(x)\right\}$. For this choice, each step $k$ of an inner loop needs to compute two proximal points $z_{k,s}$ and $x_{k,s}$.
\item Let $x^1$, $x^2$ be two variants of $x_{k,s}$ that satisfy \eqref{eq:choosexk} then their convex combination $\lambda x^1 + (1-\lambda) x^2$, where $ 0\leq \lambda\leq 1$, is also a choice of $x_{k,s}$.
\end{enumerate} 
\end{example}

It is easy to check that setting $\alpha_{1,s}=1-\alpha_3 - \frac{2}{s+\nu},~ \alpha_{2,s}=\frac{2}{s+\nu},~ {0<\alpha_{3}\leq \frac{\nu-1}{\nu+1}}$, where $\nu\geq 2$, satisfies the condition \eqref{eq:alphaupdate}.
For the update rule \eqref{eq:xsupdate}, we can choose $\tilde{x}_s=\frac{1}{m}\sum_{i=1}^m x_{i,s}$ or $\tilde{x}_s=x_{i^*,s}$, where $i^*=\argmin\limits_{i}F^P(x_{i,s}).$

\subsection{Convergence analysis}

We first give an upper bound for the variance of $v_k$ in Lemma \ref{lem:variance}.
This lemma together with the inequalities in Lemmas \ref{lem:bridge} and \ref{lem:approx} are then used to prove the upcoming Proposition \ref{prop:root}, which provides a recursive inequality within $m$ steps of an inner loop of Algorithm \ref{alg:ASM}.

\begin{lemma}
\label{lem:variance} Conditioned on $x_{k-1,s}$, we have the following expectation inequality with respect to $i_k$:
\[
\expect \norm{\nabla F(y_{k,s}) - v_k}_*^2 \leq 2 \expect\frac{1}{(nq_{i_k})^2}\norm{\grad f_{i_k}(y_{k,s}) -\grad f_{i_k}(\tilde{x}_{s-1})}_*^2+2\norm{\grad F(y_{k,s})-\grad F(\tilde{x}_{s-1}) }_*^2.
\]
\end{lemma}

\begin{lemma}
\label{lem:bridge}
At each stage $s$ of the outer loop, the following inequality holds:
\[
\begin{split}
F^P(x_{k,s}) &\leq F(y_{k,s}) + \frac{\alpha_3}{8 L_Q}\norm{\grad F(y_{k,s}) -v_k}_*^2 \\
&\quad + \alpha_{2,s}\bracket{\iprod{v_k}{z_{k,s}-z_{k-1,s}}+\theta_s D(z_{k,s},z_{k-1,s})} 
 + P(\hat{x}_{k,s}).
\end{split}
\]
\end{lemma}
We denote $\bar{z}_{k,s}=\argmin\limits_{x\in X_s} \{\iprod{v_k}{x} + P(x)+\theta_s D(x,z_{k-1,s})\}$ in \eqref{eq:proxpoint}, i.e., $\bar{z}_{k,s}$ is the exact proximal point of the randomized step $k$ at stage $s$.
\begin{lemma}
\label{lem:approx}
 For all ${ x\in X_s \cap \domain (P) }$,
\[
\begin{split}
\iprod{v_k}{z_{k,s}}+\theta_s D(z_{k,s},z_{k-1,s})&\leq \iprod{v_k}{x} + P(x)+\theta_s(D(x,z_{k-1,s})-D(x,z_{k,s}))-P(z_{k,s})+\varepsilon_{k,s} \\
&\quad-\theta_s( D(z_{k,s},\bar{z}_{k,s})-\iprod{x-z_{k,s}}{\grad h(\bar{z}_{k,s})-\grad h(z_{k,s})}). 
\end{split}
\]
\end{lemma}
\begin{proposition}
\label{prop:root}
Denote $r_{k,s}=\expect [ \alpha_{2,s}\theta_s(-D(z_{k,s},\bar{z}_{k,s})+ \iprod{x-z_{k,s}}{\grad h(\bar{z}_{k,s})-\grad h(z_{k,s})}) +\alpha_{2,s}\varepsilon_{k,s}].$
For any $x\in X_s \cap \domain(P)$, if $F^P(x)\leq F^P(x_{k,s})$, then we have:
\[
\begin{split}
\expect F^P(x_{k,s}) &\leq \alpha_{1,s} \expect F^P(x_{k-1,s}) + \alpha_{2,s} F^P(x) + \alpha_{3} F^P(\tilde{x}_{s-1}) \\
&\quad +\alpha_{2,s}^2\overline{L}( \expect D(x,z_{k-1,s}) - \expect D(x,z_{k,s})) + r_{k,s}.
\end{split}
\]
\end{proposition}
As $\tilde{x}_{s-1}$ appears in the recursive inequality, the acceleration effect works through the outer loop. The following proposition is a consequence of Proposition \ref{prop:root}. It leads to the optimal convergence rate with respect to $s$.  Theorem \ref{thrm:main} and Theorem \ref{thrm:main_inexact} indicate the rates.
\begin{proposition}\label{prop:result} 
Denote $d_{i,0}=F^P(x_{i,0}) - F^P(x^*)$ and ${ \tilde{d}_s=\expect (F^P(\tilde{x}_{s}) - F^P(x^*)) }$. Let $x^*$ be the optimal solution of \eqref{eq:composite}. Let ${r}^*_{k,s}$ be the value of $r_{k,s}$ at $x=x^*$. Suppose $x^*\in X_s$, then we have:
\[
\begin{split}
\tilde{d}_s&\leq \alpha_{2,s+1}^2\left(\frac{(1-\alpha_{2,1})d_{m,0}}{\alpha_{2,1}^2\alpha_{3}m} +\frac{1}{m\alpha_{2,1}^2}\sum\limits_{k=1}^{m-1}d_{k,0}\right.\\
&\qquad\qquad\left.+\frac{\overline{L}}{m\alpha_3}  \bracket{ \expect D(x^*,z_{m,0}) - \expect D(x^*,z_{m,s})}\right)+ \alpha_{2,s+1}^2\sum\limits_{i=1}^s\sum\limits_{k=1}^m \frac{{r}^*_{k,i}}{m\alpha_3\alpha_{2,i}^2} .
\end{split}
\]
\end{proposition}
\begin{theorem}[Convergence rate of exact ARMD] \label{thrm:main}
Suppose that ${\alpha_{2,s} \leq \frac{2}{s+2}}$, $x^* \in X_s$, and $\varepsilon_{k,s}=0$, then we have:
$$\expect (F^P(\tilde{x}_s) - F^P(x^*))=O\bracket{\frac{F^P(\tilde{x}_0)-F^P(x^*)+ \frac{4\overline{L}}{m} D(x^*,z_0)}{(s+3)^2}}.$$
\end{theorem}
\begin{remark}
In the case of exact ARMD, i.e., $\varepsilon_{k,s}=0$,  each stage $s$ of Algorithm \ref{alg:ASM} computes $(n+2m)$ gradients (if we save all gradients when computing $\tilde{v}$, then the number of gradients to be calculated at stage $s$ is only $n+m$), hence Theorem \ref{thrm:main} yields the complexity, i.e., total number of gradients computed to obtain an $\varepsilon$-optimal solution of \eqref{eq:composite} is
$O\bracket{\bracket{\frac{\sqrt{F^P(\tilde{x}_0)-F^P(x^*)+ \frac{4\overline{L}}{m} D(x^*,z_0)}}{\sqrt{\varepsilon}}-3}(n+2m)}.$
The smallest possible value of $L_Q$ is $L_Q=L_A$ when we choose $q_i=\nicefrac{L_i}{\sum_j L_j}$, i.e, the sampling probabilities of $f_i$ are proportional to their Lipschitz constants. In this case, $\overline{L}=\bracket{1+\nicefrac{4}{\alpha_3}}L_A$. Now let us choose $m=O(n)$. Hiding $F^P$ and $D(x^*,z_0)$, we can rewrite the complexity as $O\bracket{\nicefrac{(n+\sqrt{n L_A})}{\sqrt{\varepsilon}} }$. This improves upon the complexity of APG, which is $O\bracket{\nicefrac{n\sqrt{L_A}}{\sqrt{\varepsilon}}}$ (see \cite[Corrolary 1]{Tseng}). This improvement results from the fact that each step of APG must compute the  full gradient of $F(x)$ while our method only computes it periodically. In Sect.~\ref{sec:experiment}, we show experiments comparing exact ARMD with APG that verify this property.
\end{remark}
%
%To establish convergence rate of inexact ARMD, we consider 2 cases.
%
%\begin{assumption}
%\label{asstn:bounded}
%The sequences $\{z_{k,s}\}$ generated by Algorithm \ref{alg:ASM} are bounded, i.e, there exists a constant $C$ such that  $\norm{z_{k,s}}\leq C.$
%\end{assumption}
%
\begin{theorem}[Convergence rate of inexact ARMD] \label{thrm:main_inexact}
Suppose that $\alpha_{2,s} \leq \frac{2}{s+2}$, $x^* \in X_s$, and $h(\cdot)$ is $L_h$-Lipschitz smooth. Let $\left\{\epsilon_s\right\}$ be a sequence of nonnegative numbers that satisfies $\sum\limits_{i=1}^\infty \sqrt{\frac{\epsilon_i}{\alpha_{2,i}}}$ is convergent. 
\begin{itemize}
\item[(i)] Assume that there exists a constant $C$ such that $\norm{z_{k,s}}\leq C$ (for example, if $P$ is the indicator function  of a bounded closed convex set, then its domain is naturally bounded; thus, as $z_{k,s}\in \domain(P)$  we have $\left\{z_{k,s}\right\}$ is bounded) and we let the error $\varepsilon_{k,s}$ in \eqref{eq:proxpoint} be $\varepsilon_{k,s}=\epsilon_s$ for $k=1,\ldots,m$, then there exists a constant $\bar{C}=O(C L_h \sqrt{\bar{L}})$ such that
\begin{equation}
\label{eq:rateinexact}\expect (F^P(\tilde{x}_s) - F^P(x^*))=O\bracket{\frac{F^P(\tilde{x}_0)-F^P(x^*)+ \frac{4\overline{L}}{m} D(x^*,z_0) + \bar{C}}{(s+3)^2}}
\end{equation}
\item[(ii)] Assume that $P$ is bounded below (this assumption is satisfied in many practical problems, for instances, the examples in the introduction section satisfying $P(x)\geq 0, \forall x\in \mathcal{E}$), and we use the following adaptive update rule 
\begin{equation}
\label{eq:adaptive}
\max\left\{\norm{\bar{z}_{k,s}} ^2\varepsilon_{k,s},C\varepsilon_{k,s}\right\}\leq C \epsilon_s,
\end{equation}
where $C$ is some constant, then there also exists a constant $\bar{C}=O(CL_h \sqrt{\bar{L}})$ such that \eqref{eq:rateinexact} is satisfied. 
\end{itemize}
\end{theorem}
\begin{remark}
The type of the inexact proximal point in \eqref{eq:proxpoint} has been first considered in \cite{Auslender1987}. It can be checked by verifying the duality gap while solving the minimization problem in  \eqref{eq:proxpoint}, see e.g., \cite{Schmidt2011,Villa}. To use the adaptive update rule \eqref{eq:adaptive}, we  give an upper bound for  $\norm{\bar{z}_{k,s}}$ before computing the inexact proximal point in \eqref{eq:proxpoint} as follows. Without lost of generality, we assume that  $P(x)\geq  0, \forall x \in  \mathcal{E}$ . Since  $\bar{z}_{k,s}$ is the exact proximal point of \eqref{eq:proxpoint},
\[\iprod{v_k}{\bar{z}_{k,s}}+P(\bar{z}_{k,s})+\theta_s D(\bar{z}_{k,s},z_{k-1,s})\leq \iprod{v_k}{z_{k-1,s}} + P(z_{k-1,s}).
\] 
This implies $\frac{\theta_s}{2}\norm{\bar{z}_{k,s}-z_{k-1,s}}^2 \leq \norm{v_k}_* \norm{\bar{z}_{k,s}-z_{k-1,s}} + P(z_{k-1,s}).$
Hence, we have 
\[\norm{\bar{z}_{k,s}}\leq \norm{\bar{z}_{k,s}-z_{k-1,s}} + \norm{z_{k-1,s}}\leq \frac{\norm{v_k}_*}{\theta_s}+\sqrt{\frac{\norm{v_k}^2_*}{\theta_s^2}+\frac{2P(z_{k-1})}{\theta_s}}+ \norm{z_{k-1,s}}. 
\]
\end{remark}
\begin{remark}
 In the case of inexact ARMD, we require the series $\sum\limits_{i=1}^\infty \sqrt{\frac{\epsilon_i}{\alpha_{2,i}}}$ to be convergent to obtain the convergence rate $O(1/s^2)$. A sufficient condition is that  $\epsilon_s$ decreases as  $O(1/s^{3+\delta})$ for $\delta>0$. It is worth mentioning that for the deterministic inexact proximal-gradient (PG) method considered in \cite{Schmidt2011}, the sequence of errors $\left\{\epsilon_s\right\}$ is required to decrease as  $O(1/s^{2+\delta})$ for the basic PG method (see \cite[Proposition 1]{Schmidt2011}), and as $O(1/s^{4+\delta})$ for the accelerated PG method (see \cite[Proposition 2]{Schmidt2011}). Comparing with our inexactness condition in Theorem \ref{thrm:main_inexact}(i), the inexactness of $\left\{\epsilon_s\right\}$ is one order more than in the basic PG method, but it is one order weaker  than the accelerated PG method. This happens unsurprisingly since although our scheme applies acceleration technique for the inner loop, it still computes full gradients for the outer loop. 
\end{remark}

\subsection{Extension to Nonsmooth Problems}
\label{subsec:nonsmooth}
Problem \eqref{eq:composite} with non-smooth $f_i$ often arises in statistics and machine learning. One well-known example is the $l_1$-SVM:
\begin{equation}\label{eq:SVM}
\min\limits_x\left\{\frac{1}{n}\sum\limits_{i=1}^n \max\{1-b_i\iprod{a_i}{x},0\} + \lambda \norm{x}_1\right\}.
\end{equation}
If $f_i(x)$ in \eqref{eq:composite} is non-smooth, we smooth it by $f_{\mu,i}$ and then solve the problem:
\begin{equation} \label{eq:smoothprob}
\min\limits_{x\in X}F^P_\mu(x):=\frac{1}{n}\sum\limits_{i=1}^n f_{\mu,i}(x)+P(x),
\end{equation}
where we assume the smooth functions $f_{\mu,i}$ satisfy the following assumptions for $\mu>0$.
\begin{assumption}
\label{asstn:smooth}
\begin{itemize}
\item[(1)] Functions $f_{\mu,i}$ are convex,  $L_{\mu,i}$-Lipschitz smooth, and 
\item[(2)] There exist constants $\overline{K}_i,\underline{K}_i\geq 0$ such that $f_{\mu,i}(x) -\underline{K}_i\mu \leq  f_i(x) \leq f_{\mu,i}(x) +  \overline{K}_i \mu$, $\forall x$.
\end{itemize}
\end{assumption}
Note that it is not necessary to smooth $P(x)$, and this can be useful in practice. For example, if $P(x)$ is the $\ell_1$ norm, then smoothing it undermines its ability to obtain sparse solutions.
We now give some examples of popular smoothing functions that satisfy Assumption \ref{asstn:smooth}.
\begin{example}
\label{ex:smooth}
\begin{itemize}
\item Let $A$ be a matrix in $\mathbb{R}^{q\times p}$, $u\in \mathbb{R}^q$, $U$ be a closed convex bounded set, and $f(x)=\max\limits_{u\in U}\left\{ \iprod{Ax}{u}-\hat{\phi}(u)\right\}$, where $\hat{\phi}(u)$ is a continuous convex function. Let $R(u)$ be a continuous strongly convex function with constant $\sigma>0$ and $\norm{A}$ be the spectral norm of $A$. Nesterov in \cite{Nesterov} proved that the convex function $f_\mu(x)=\max\limits_{u\in U} \left\{\iprod{Ax}{u}-\hat{\phi}(u) -\mu R(u)\right\}$ is a $L_\mu$-Lipschitz smooth function of $f(x)$ with $L_\mu=\frac{\norm{A}^2}{\mu \sigma}$. We easily see that $f(x)-K \mu \leq f_\mu(x) \leq f(x)$ for $\mu>0$, where $K=\max\limits_{u\in U} R(u)$. Hence $f_\mu(x)$ satisfies Assumption \ref{asstn:smooth}.
\item One smoothing function of $[x]_+=\max\{x,0\}$ is $f^{(a)}_\mu(x)=\frac12\bracket{x+\sqrt{x^2+4\mu^2}}$, which is convex and $L_\mu$-Lipschitz smooth with $L_\mu=\frac{1}{4\mu}$. Since $[x]_+\leq f^{(a)}_\mu(x)\leq [x]_+ + \mu$, Assumption \ref{asstn:smooth} is satisfied. Another smoothing function of $[x]_+$ is the neural network smoothing function $f^{(b)}_\mu(x)=\mu\log\bracket{1+e^{\nicefrac{x}{\mu}}}$, which is widely used in SVM problems \cite{Lee2001}. It is not difficult to prove that $[x]_+ \leq f^{(b)}_\mu(x)\leq [x]_+ + (\log2)\mu$ and $f^{(b)}_\mu(x)$ is $L_\mu$-Lipschitz smooth with $L_\mu=\frac{1}{4\mu}$. Hence, $f^{(b)}_\mu(x)$ satisfies Assumption \ref{asstn:smooth}. Using these smoothing functions of $[x]_+$, we can smooth the $l_1$-SVM \eqref{eq:SVM} to get the form \eqref{eq:smoothprob}.
\end{itemize}
\end{example}
%
%The following Theorem proves that if we choose appropriate smoothing parameter $\mu$, we can achieve an $\varepsilon$-optimal solution of \eqref{eq:composite} after running $O\bracket{\frac{1}{\sqrt{\varepsilon}}}$ stages of Algorithm \ref{alg:ASM}. 
We now state the main theorem for this non-smooth case.
\begin{theorem} \label{thrm:smoothcase}
Denote $\overline{K}=\frac{1}{n}\sum\limits_{i=1}^n \overline{K}_i$ and $\underline{K}=\frac{1}{n}\sum\limits_{i=1}^n \underline{K}_i$. Let $\left\{\tilde{x}_{\mu,s}\right\}$ be the sequence generated by Algorithm \ref{alg:ASM} for solving Problem \eqref{eq:smoothprob}, $\bar{L}_\mu$ be the corresponding value of $\bar{L}$ in Algorithm \ref{alg:ASM} for Problem \eqref{eq:smoothprob}, and $x^*$ be the solution of \eqref{eq:composite}. We have:
$
\expect F^P(\tilde{x}_{\mu,s})-F^P(x^*)=O\bracket{\frac{1+\frac{4\overline{L}_\mu}{m}+\bar{C}}{(s+3)^2}} + \bracket{\overline{K}+\underline{K}} \mu,
$
where $\bar{C}=0$ if the involving proximal points are calculated exactly, and $\bar{C}=O\bracket{\sqrt{\bar{L}_\mu}}$ if the involving proximal points are calculated inexactly under the assumptions of Theorem \ref{thrm:main_inexact}.
\end{theorem}
This theorem shows that if we choose $\mu< \frac{\epsilon}{\overline{K}+\underline{K}}$  and suppose that $\bar{L}_\mu=O(1/\mu)$, then after running $s$ stages of Algorithm \ref{alg:ASM} with 
\[ 
s > \frac{\sqrt{1+\frac{4\overline{L}_\mu}{m}}}{\sqrt{\epsilon -(\overline{K}+\underline{K})\mu}}=O\bracket{\frac{1}{\epsilon^{1/2}}+\frac{1}{\sqrt{m}\epsilon}}
\]
for the case of exact proximal points, or
\[
s > \frac{\sqrt{1+\frac{4\overline{L}_\mu}{m}+\sqrt{\bar{L}_\mu}}}{\sqrt{\epsilon -(\overline{K}+\underline{K})\mu}}=O\bracket{\frac{1}{\epsilon^{3/4}}+\frac{1}{\epsilon^{1/2}}+\frac{1}{\sqrt{m}\epsilon}}
\]
for the case of inexact proximal points,  we can obtain an $\epsilon$-optimal solution of \eqref{eq:composite}. We remark that the condition  $\bar{L}_\mu=O(1/\mu)$ is satisfied by examples in Example \ref{ex:smooth}.
Problem \eqref{eq:composite} with non-smooth $f_i$ can also be solved by subgradient methods or stochastic subgradient methods with the convergence rate $ O(\frac{1}{\epsilon^2})$ \cite{Nesterov2004}, which is significantly slower than the rate obtained using our smoothing scheme. In terms of complexity (the total number of gradients computed to achieve an $\epsilon$-optimal solution), if $m=O(n)$, our smoothing scheme computes $O\bracket{n/\epsilon^{1/2}+\sqrt{n}/\epsilon}$ gradients for the exact proximal points case, and $O\bracket{n/\epsilon^{3/4}+n/\epsilon^{1/2}+\sqrt{n}/\epsilon}$ for the inexact case. These complexities definitely outperform the complexity $O\bracket{n/\epsilon^2}$ of subgradient methods.

\section{Experiments}
\label{sec:experiment}

\subsection{Experiment with Synthetic Datasets}
\label{sec:synthetic_exp}

We consider the Lasso problem with $f_i(x) = \frac12 (\iprod{a_i}{x} - b_i)^2 $ and $ P(x) = \lambda \norm{x}_1$. 
The problem is non-strongly convex and can also be solved by FISTA \cite{Beck2009}, SAGA \cite{Defazio_Bach} and APG \cite{Tseng}.
Proximal SVRG \cite{Xiao_Zhang} can also be applied to this problem, although we need to add a regularizer $\frac{c}{2} \norm{x}^2_2$ to $P(x)$ to maintain the strong convexity and hence the convergence property of this algorithm. 
For simplicity, we set $c = 0.001$ in this experiment.
% Since SAGA is much slower than FISTA and APG due to its $O(1/k)$ convergence rate, we only compare ARMD with FISTA and APG.

We generate 9 synthetic datasets with sizes $n = 1000, 10000, 50000$ and with $p = 10, 100, 500$ dimensions as follows.
First, we generate $n$ vectors $a_1, \ldots, a_N$ uniformly on $[0,10]^p$ and a true sparse target vector $x^* \in \{ 0, 1\}^p$ with a random half of its elements being 0 and the rest having value 1.
For each $a_i$, we generate $b_i = \iprod{a_i}{x^*} + \epsilon$ where $\epsilon \sim \mathcal{N}(0, 0.01^2)$ is a small Gaussian noise. We set $\lambda=0.1$  and use the same settings of ARMD on all the datasets.
We test two versions of exact ARMD based on different updating rules of $x_{k,s}$ in Algorithm \ref{alg:ASM}. 
Particularly, ARMD I and ARMD II respectively denote ARMD when using $x_{k,s} = \hat{x}_{k,s}$ and 
$x_{k,s} =\argmin\limits_x\left\{\iprod{v_k }{x} + \frac{\overline{L}}{2}\norm{x-y_{k,s}}^2 + P(x)\right\}$
in \eqref{eq:choosexk} (see Example \ref{ex:xksupdate}).
For both versions, we further test on two cases:
\[
(1) \quad \alpha_{2,s}=\frac{2}{s+2}, \alpha_{3}= \frac13; \qquad \text{ and } \qquad (2) \quad \alpha_{2,s}=\frac{2}{s+5},  \alpha_{3}= \frac23.
\]
For Lasso, the Lipschitz constant of $f_i$ is $L_i=a_i^2$. 
We use $D(x,y)=\frac12\norm{x-y}^2$ in ARMD and thus $z_{k,s}$ can be computed by soft thresholding.
We also set $q_i=\frac{1}{n}$ and $m=n$.
Note that ARMD reduces the complexity compared to the other algorithms.
So we evaluate the performance based on the objective function value and the optimality gap v.s. $\frac{\#  \text{ of computed gradients}}{n}$.
From the plots in Fig.~\ref{fig_synthetic}, the ARMD methods are significantly better than FISTA, SAGA, and APG in all datasets.
Proximal SVRG performs comparable to ARMD, although in all cases it converges to a worse optimal value than ARMD II with $\alpha_3 = 1/3$ (see Fig.~\ref{fig_synthetic}(b)).
This justifies that ARMD requires much less gradient computations to converge than the other algorithms.
From Fig.~\ref{fig_synthetic}(b), ARMD II with $\alpha_3 = 1/3$ consistently performs the best among different ARMD versions.

\begin{figure*}[!t]
\centering
\begin{subfigure}[b]{\textwidth}
    \centering
    \includegraphics[width=0.3\textwidth]{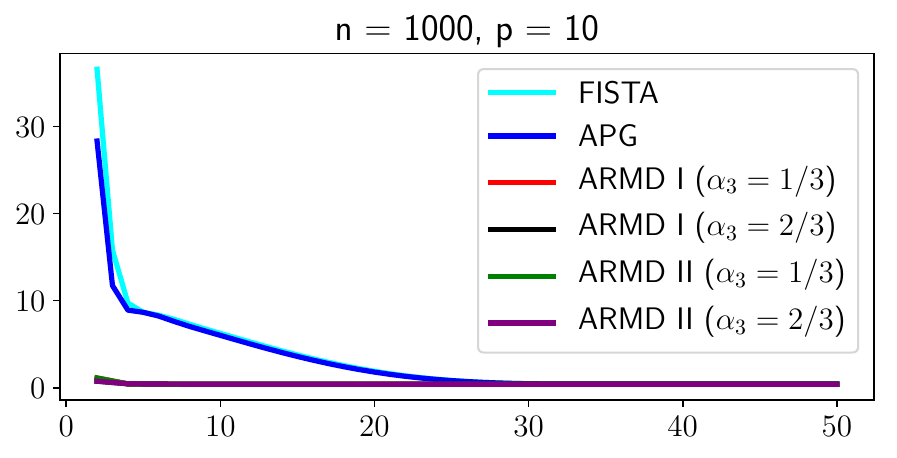}
    \includegraphics[width=0.3\textwidth]{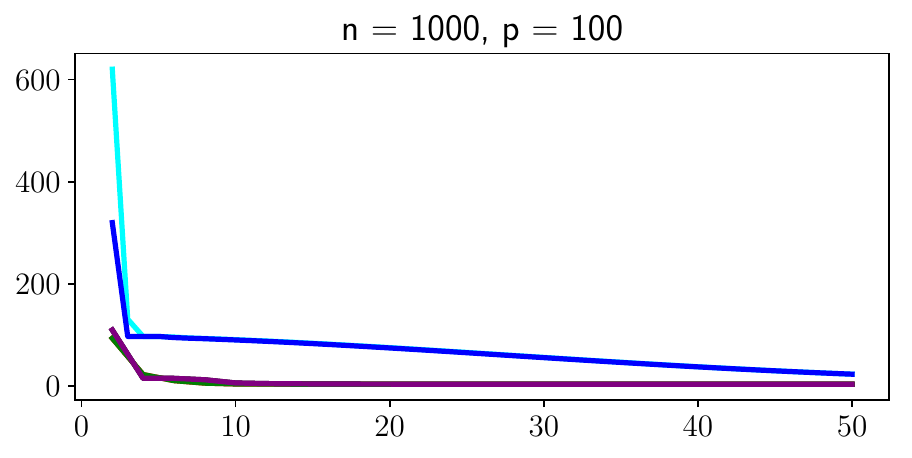}
    \includegraphics[width=0.3\textwidth]{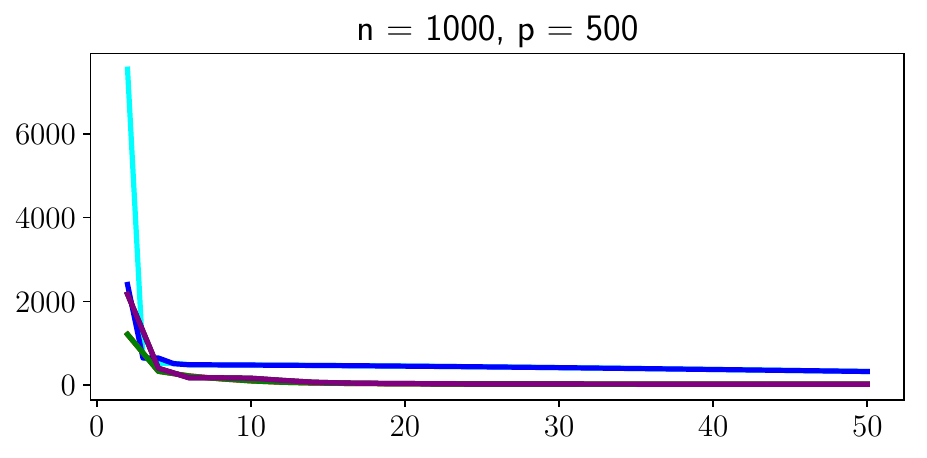}
    \\
    \includegraphics[width=0.3\textwidth]{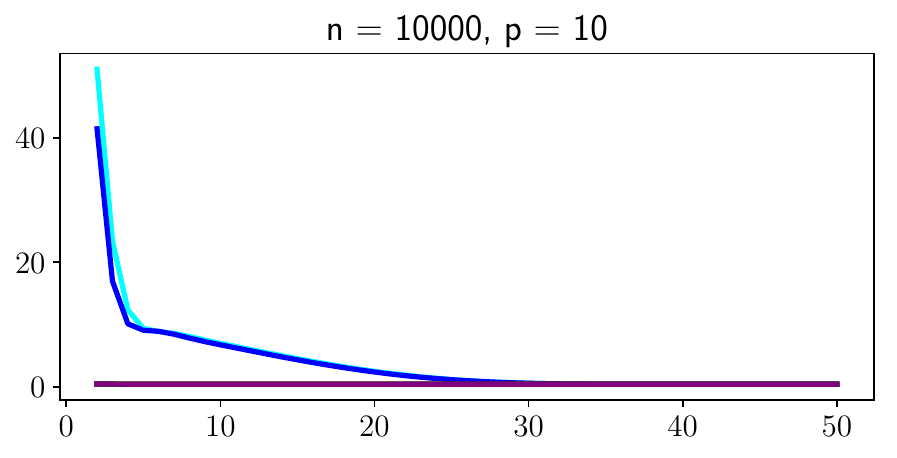}
    \includegraphics[width=0.3\textwidth]{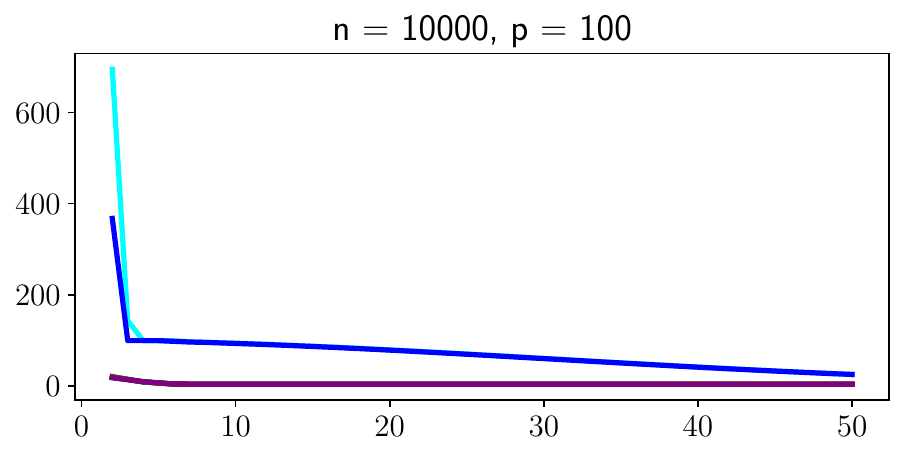}
    \includegraphics[width=0.3\textwidth]{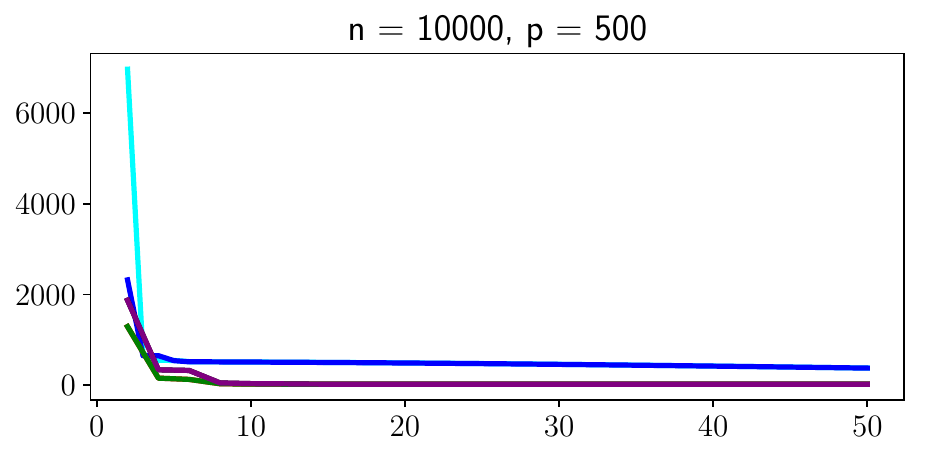}
    \\
    \includegraphics[width=0.3\textwidth]{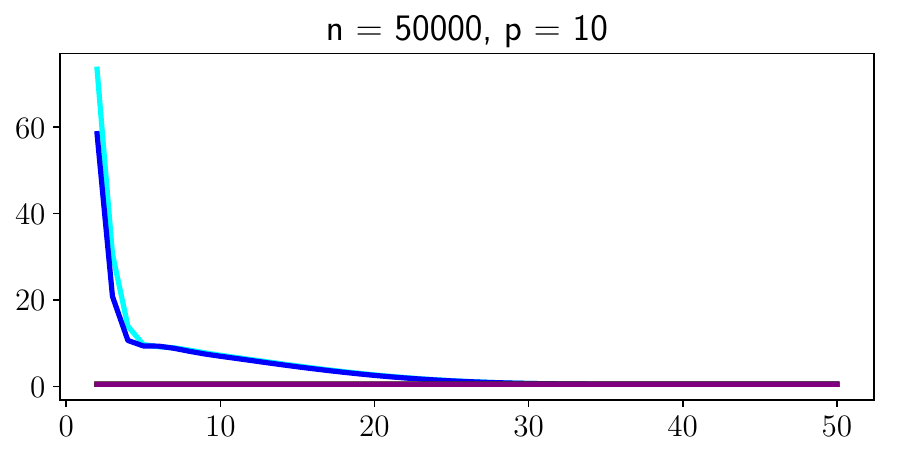}
    \includegraphics[width=0.3\textwidth]{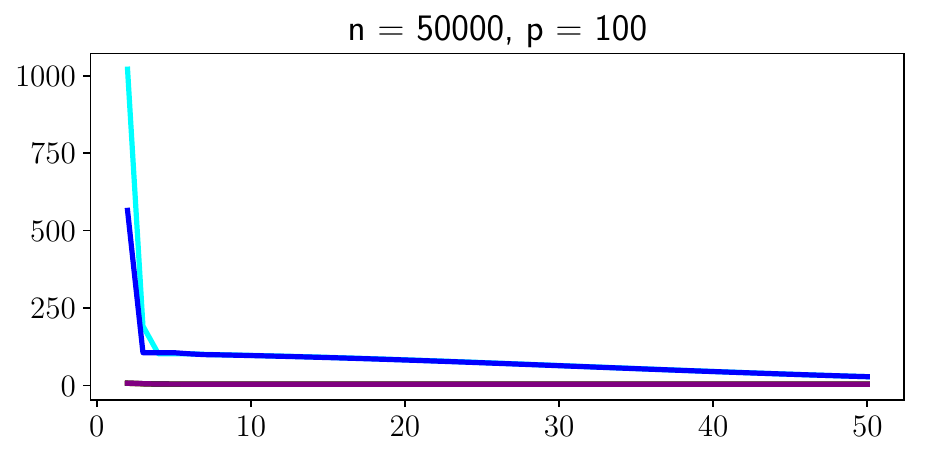}
    \includegraphics[width=0.3\textwidth]{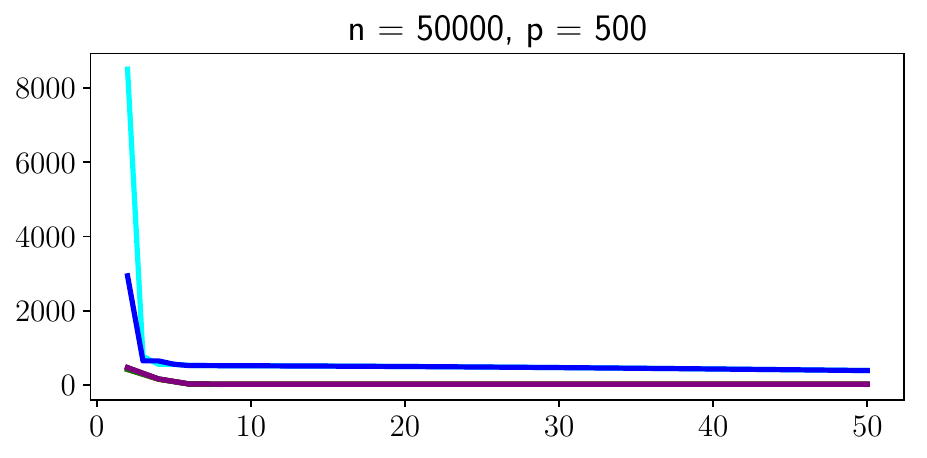}
\vspace{-3mm}
\caption{objective value}
\end{subfigure}

\begin{subfigure}[b]{\textwidth}
    \centering
    \includegraphics[width=0.3\textwidth]{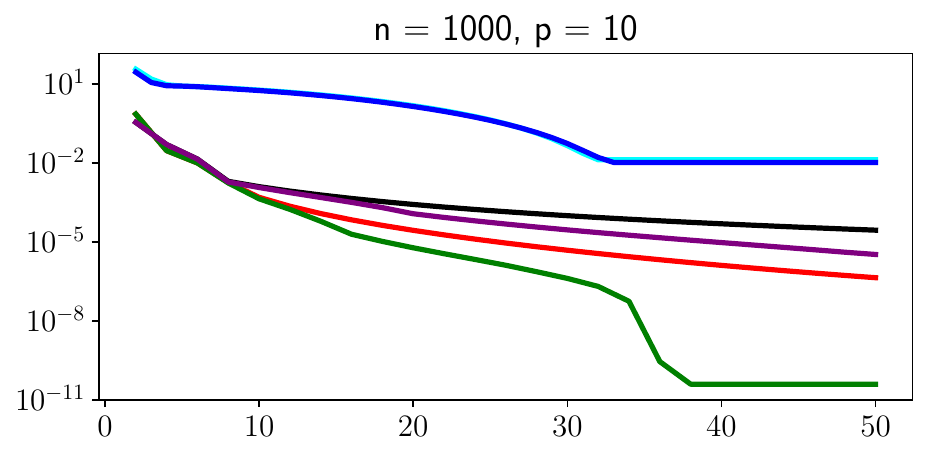}
    \includegraphics[width=0.3\textwidth]{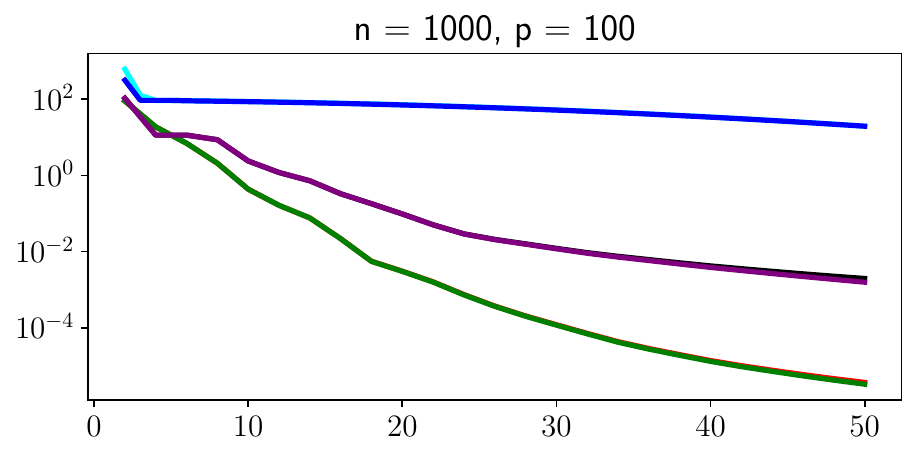}
    \includegraphics[width=0.3\textwidth]{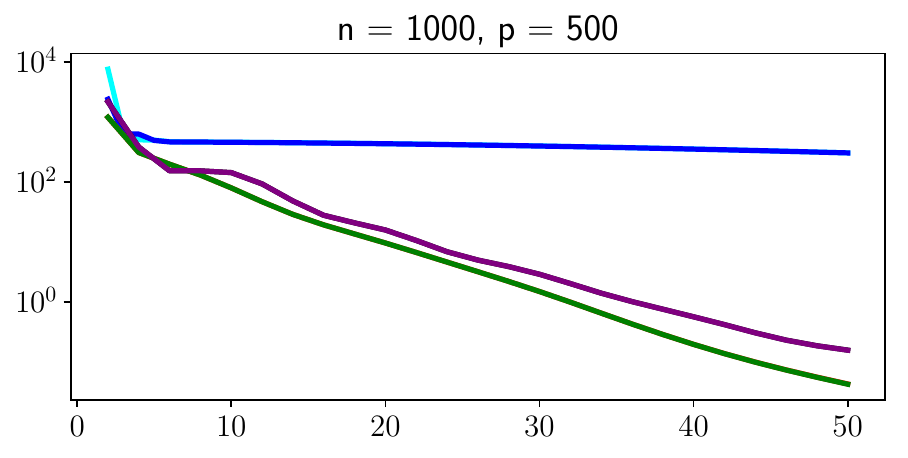}
    \\
    \includegraphics[width=0.3\textwidth]{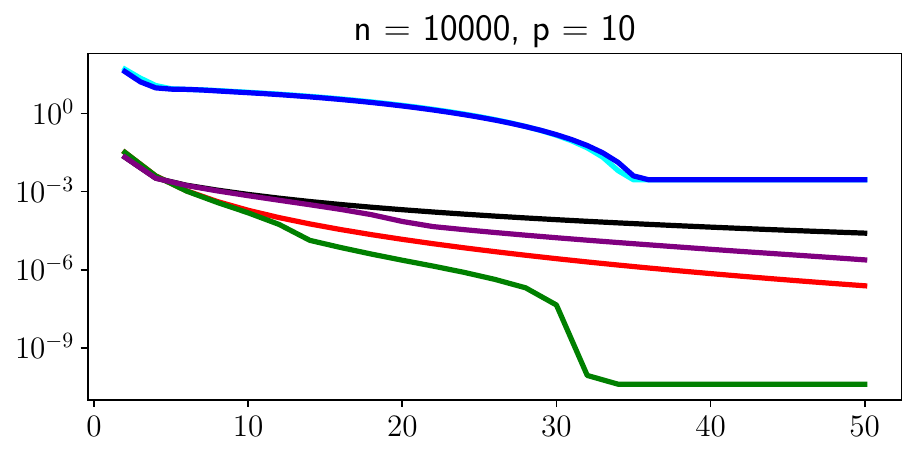}
    \includegraphics[width=0.3\textwidth]{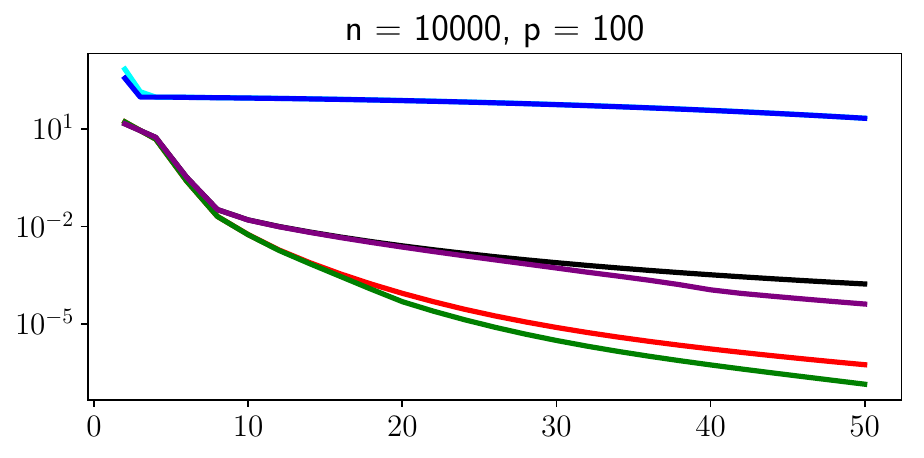}
    \includegraphics[width=0.3\textwidth]{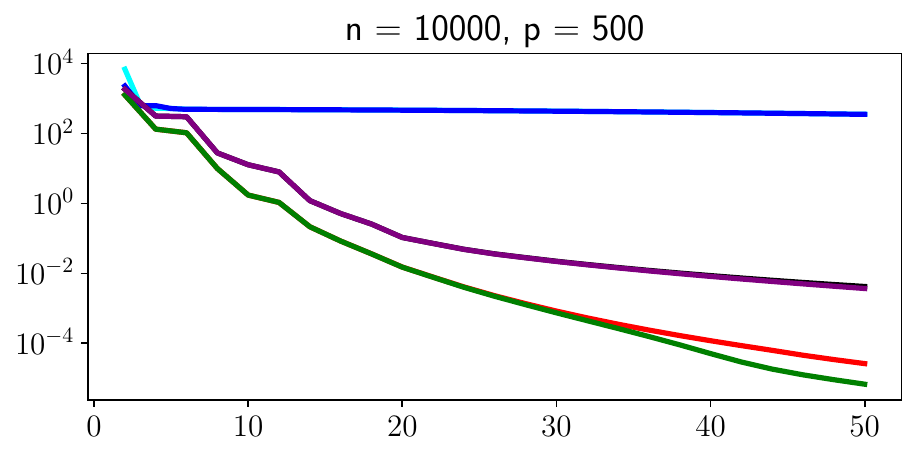}
    \\
    \includegraphics[width=0.3\textwidth]{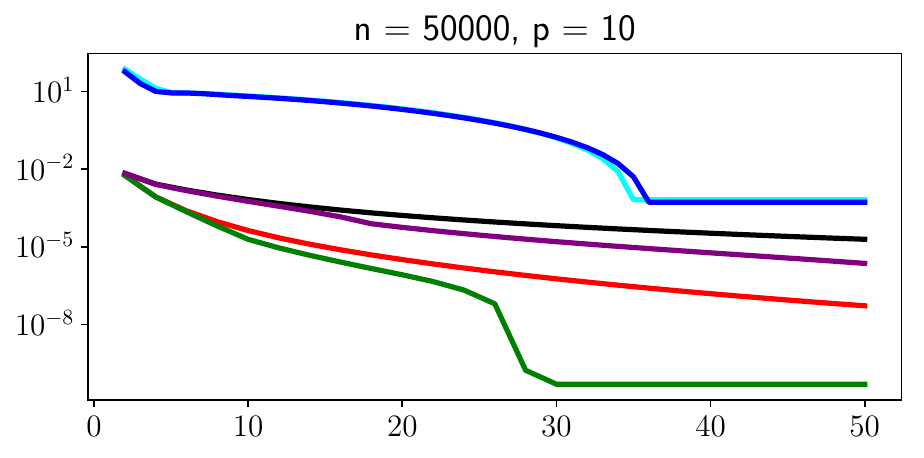}
    \includegraphics[width=0.3\textwidth]{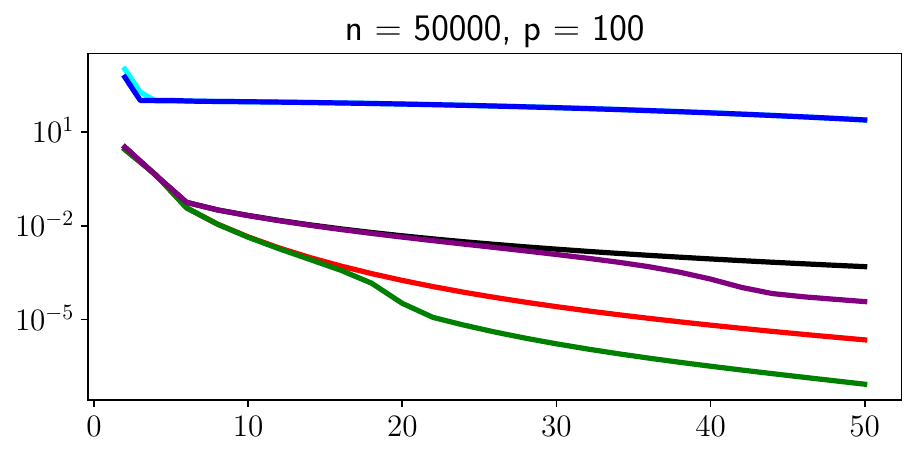}
    \includegraphics[width=0.3\textwidth]{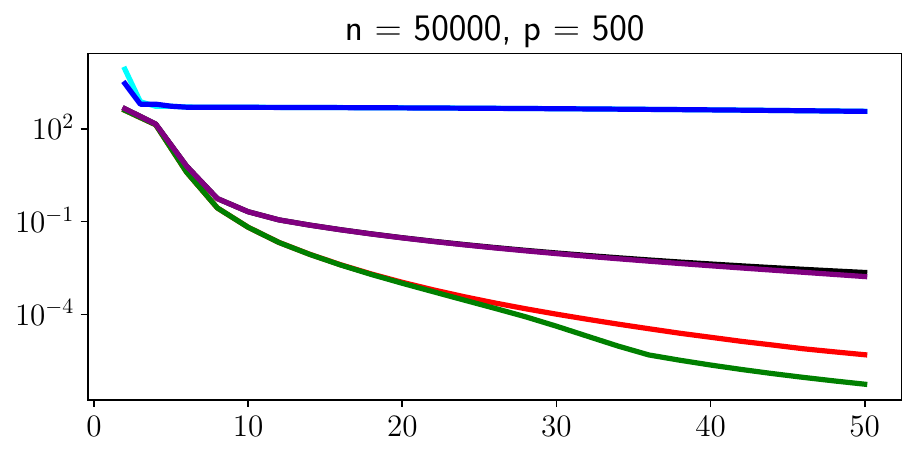}
\vspace{-3mm}
\caption{optimality gaps}
\end{subfigure}
\vspace{-0.6cm}
\caption{Comparison of FISTA, AGP, and exact ARMD for Lasso problem on synthetic datasets. The vertical axis is the objective value (or optimality gap) and the horizontal axis is the number of computed gradients$/n$.}
\label{fig_synthetic}
\end{figure*}

\subsection{Experiments with Real Datasets}

\textbf{Exact ARMD.}
%\label{sec:asmd_exp}
We consider the same Lasso problem and settings as in Sect.~\ref{sec:synthetic_exp} on 6 real datasets from \cite{libsvmdata}: breast-cancer ($n=683, p=10$), abalone ($n=4177, p=8$), mushrooms ($n=8124, p=112$), cpusmall ($n=8192, p=12$), letter ($n=15000,p=16$), and cadata ($n=20640,p=8$).
From Fig.~\ref{fig_resall}, the ARMD methods are also significantly better than FISTA, SAGA, and APG; and  ARMD II with $\alpha_3 = 1/3$ also performs the best among all algorithms on the datasets except for the cadata dataset, where ARMD II with $\alpha_3 = 2/3$ is the best method.
 
\begin{figure*}[!t]
\centering
\begin{subfigure}[b]{\textwidth}
    \centering
    \includegraphics[width=0.3\textwidth]{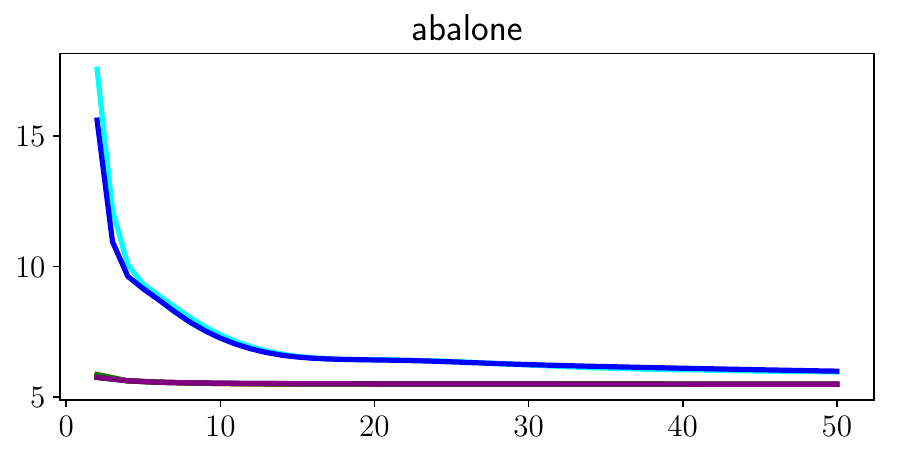}    
    \includegraphics[width=0.3\textwidth]{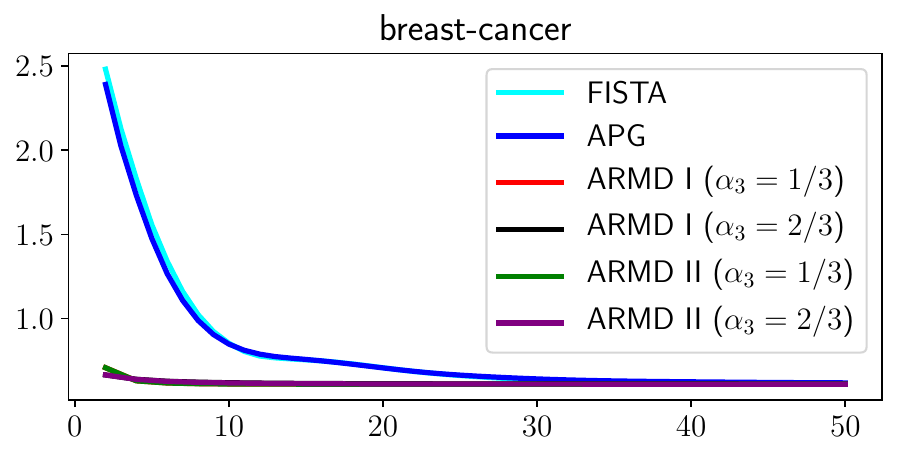}
    \includegraphics[width=0.3\textwidth]{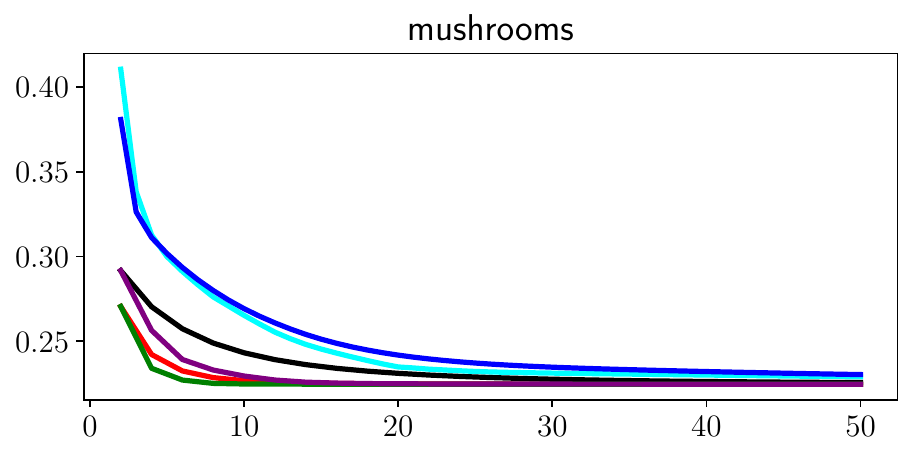}
    \\
    \includegraphics[width=0.3\textwidth]{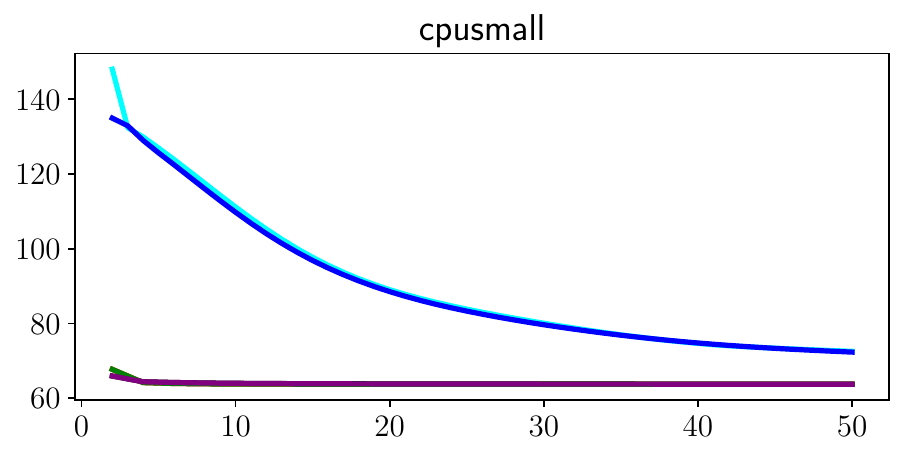}
    \includegraphics[width=0.3\textwidth]{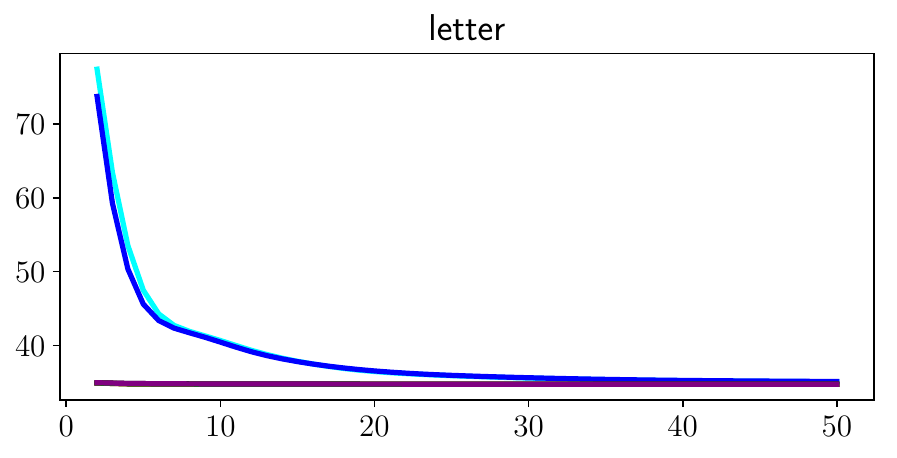}
    \includegraphics[width=0.3\textwidth]{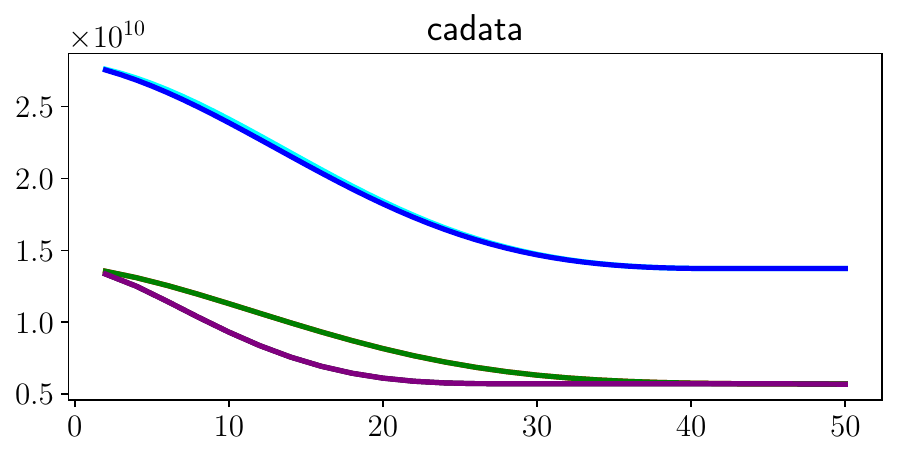}
\vspace{-3mm}
\caption{objective value}
\end{subfigure}

\begin{subfigure}[b]{\textwidth}
    \centering
    \includegraphics[width=0.3\textwidth]{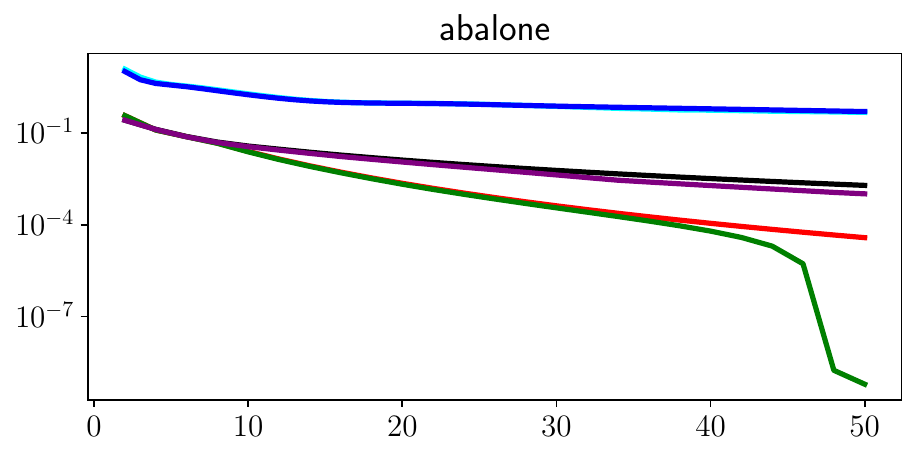}    
    \includegraphics[width=0.3\textwidth]{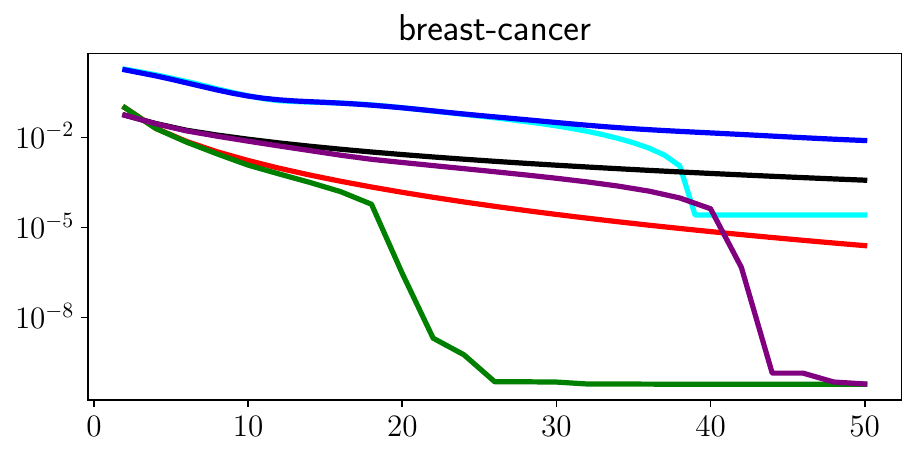}
    \includegraphics[width=0.3\textwidth]{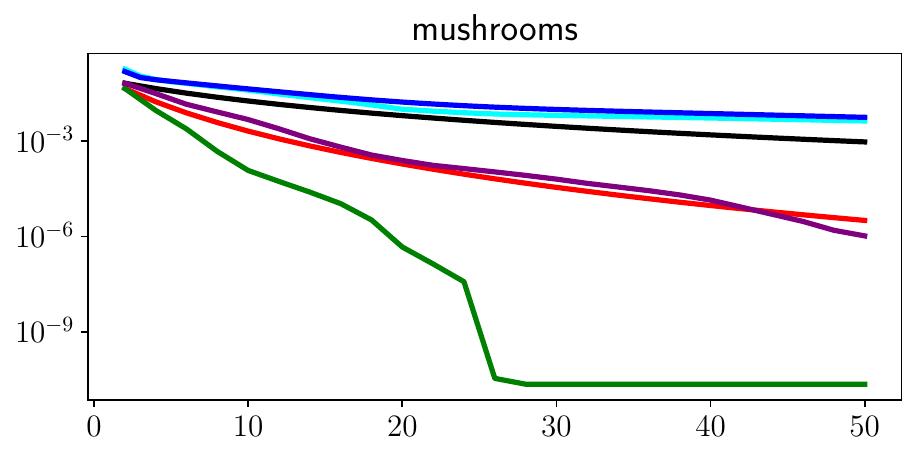}
    \\
    \includegraphics[width=0.3\textwidth]{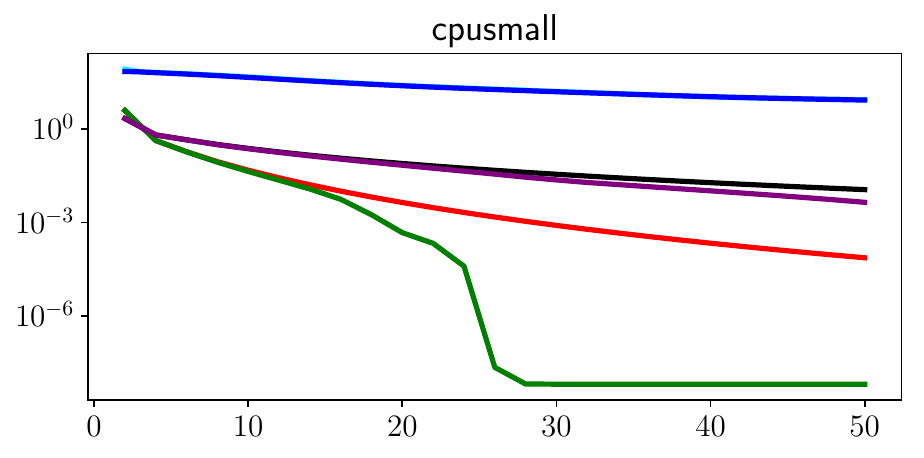}
    \includegraphics[width=0.3\textwidth]{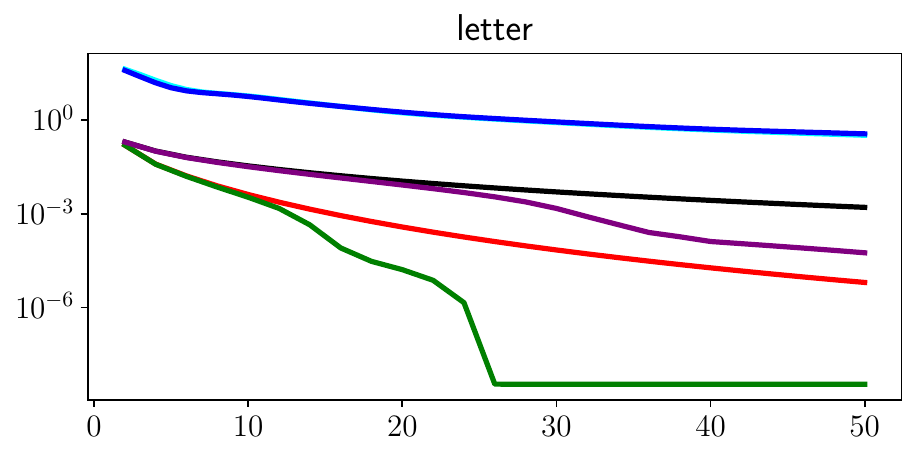}
    \includegraphics[width=0.3\textwidth]{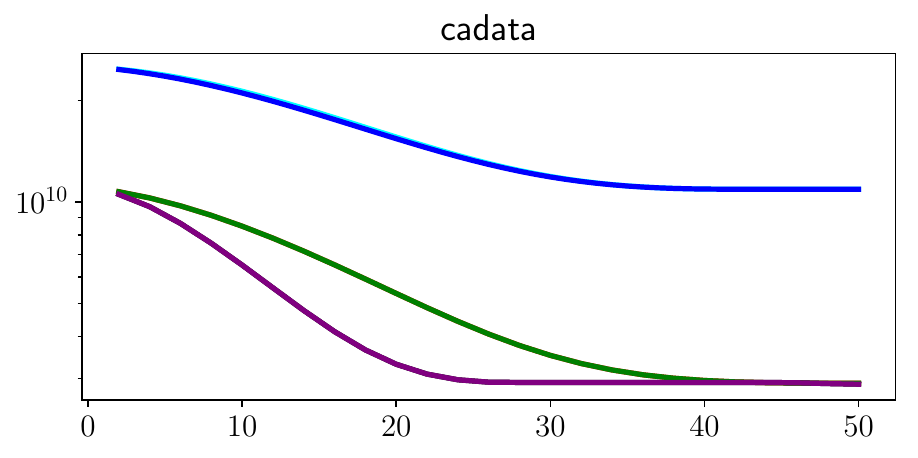}
\vspace{-3mm}
\caption{optimality gaps}
\end{subfigure}

\vspace{-4mm}
\caption{Results on real datasets for exact ARMD (same settings as Fig.~\ref{fig_synthetic}).}
\label{fig_resall}
\end{figure*}

\noindent
\textbf{Inexact ARMD:}
We consider the overlapping group Lasso problem \cite{jacob2009group}, which is also an instance of \eqref{eq:composite} with $ f_i(x) = \frac12 (\iprod{a_i}{x} - b_i)^2 $ and $ P(x) = \lambda \Omega_{\text{overlap}}^{\mathcal{G}}(x) $.
For a collection of overlapping groups $\mathcal{G} = \{ G_r \}_{r=1}^B$ where $G_r \subseteq \{ 1, \ldots, p \}$ and $p$ is the number of dimensions of $x$, the penalty $\Omega_{\text{overlap}}^{\mathcal{G}}(x)$ is defined as:
\begin{equation}
\label{eq:overlap-reg}
\textstyle
\Omega_{\text{overlap}}^{\mathcal{G}}(x) = \inf_{(v_1,\ldots,v_B), v_r \in \mathbb{R}^p, \text{supp}(v_r) \subseteq G_r, \sum_{r=1}^B v_r = x} \left\{ \sum_{r=1}^B \norm{v_r}_2 \right\}.
\end{equation}
For this penalty, the proximal points cannot be computed exactly in a finite number of steps \cite{mosci2010primal}, thus we need to use inexact methods to compute the proximal points up to an error $\varepsilon_s$.
In this experiment, we use the same real datasets and settings above, except that the proximal points $z_{k,s}$ and $x_{k,s}$ for ARMD II are estimated using the projection method proposed in \cite{mosci2010primal}.
We set $\mathcal{G} = \{ \{ 1, 2, 3\}, \{ 3, 4, 5\}, \{ 5, 6, 7 \}, \ldots \}$, $\lambda = 0.1$, $\varepsilon_s = 0.01/s^{4.001}$, and compare inexact ARMD methods with FISTA, APG, SAGA, and SVRG for overlapping group Lasso \cite{mosci2010primal}.
The value of $\epsilon_s$ is chosen to ensure the convergence of APG (which requires $\epsilon_s = O(1/s^{4+\delta})$, see \cite{Schmidt2011}) as well as ARMD (which requires $\varepsilon_s = O(1/s^{3+\delta})$).%The factor $10^{-3}$ in $\varepsilon_s$ is added to make sure the errors of the first few iterations are not too large.
The penalty $\Omega_{\text{overlap}}^{\mathcal{G}}(x)$ is computed at any point $x$ by solving Problem \eqref{eq:overlap-reg}.

Fig.~\ref{fig_resall_iasmd} plots the objective value and optimality gap against $\frac{\# \text{ of computed gradients}}{n}$.
From Fig.~\ref{fig_resall_iasmd}(a), FISTA reduces the objective value slightly faster than other algorithms during the first few iterations in many cases, but it eventually converges to worse optimal values than ARMD methods.
Regarding to optimality gaps,  ARMD methods achieve better optimality gaps than the other algorithms in all except for the mushrooms dataset, where SVRG performs better.

\begin{figure*}[!t]
\centering
\begin{subfigure}[b]{\textwidth}
    \centering
    \includegraphics[width=0.3\textwidth]{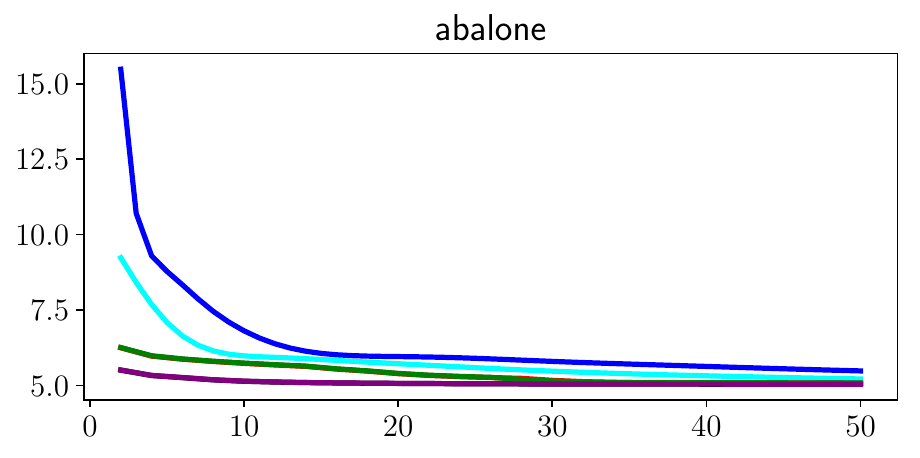}    
    \includegraphics[width=0.3\textwidth]{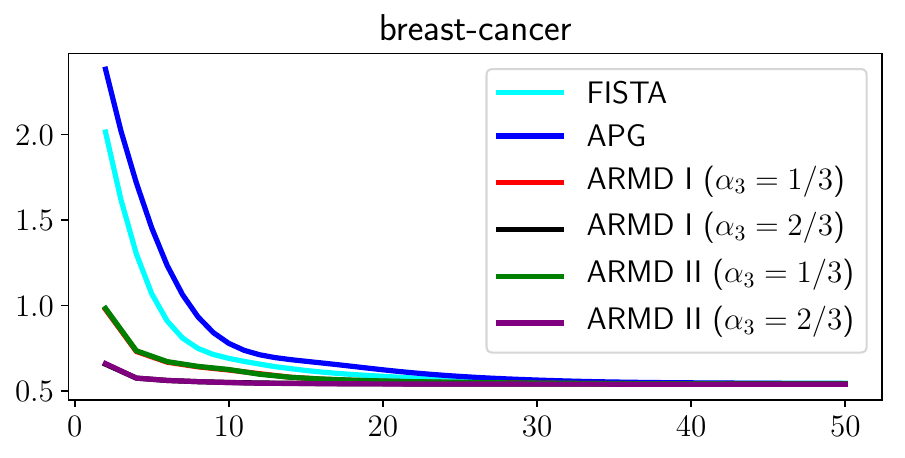}
    \includegraphics[width=0.3\textwidth]{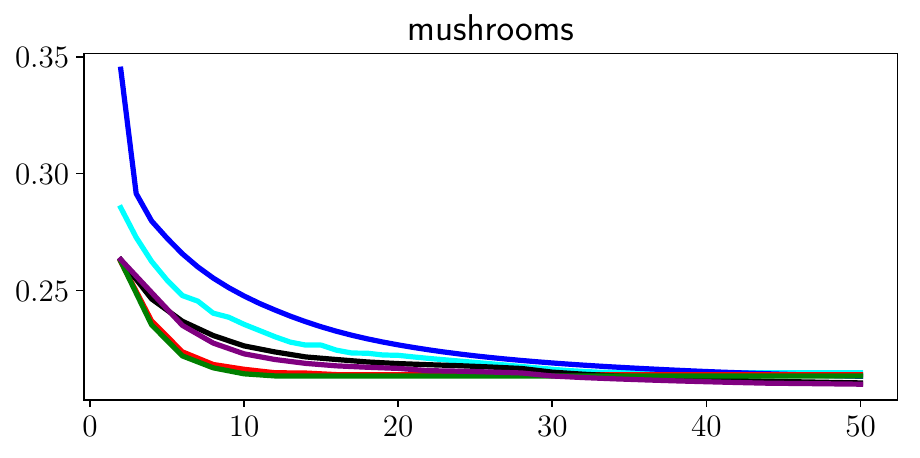}
    \\
    \includegraphics[width=0.3\textwidth]{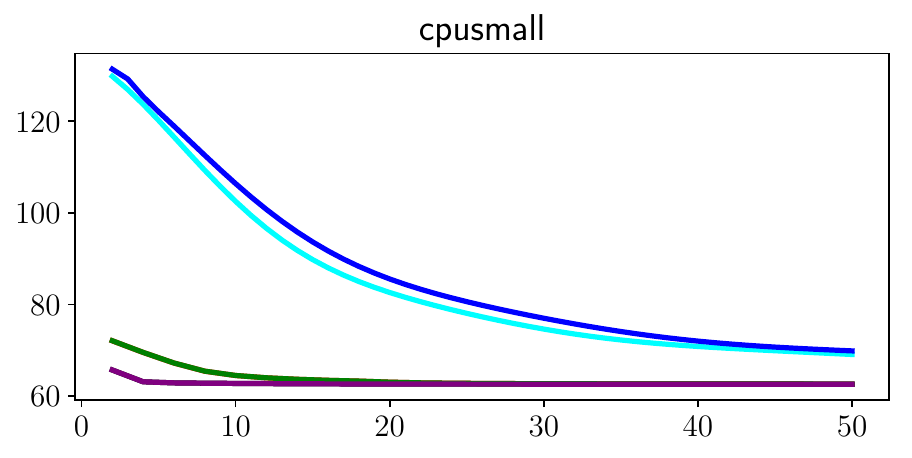}
    \includegraphics[width=0.3\textwidth]{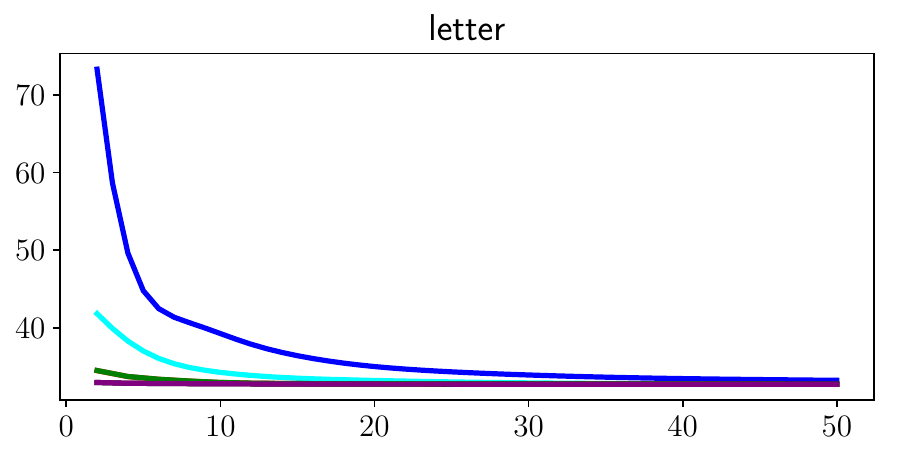}
    \includegraphics[width=0.3\textwidth]{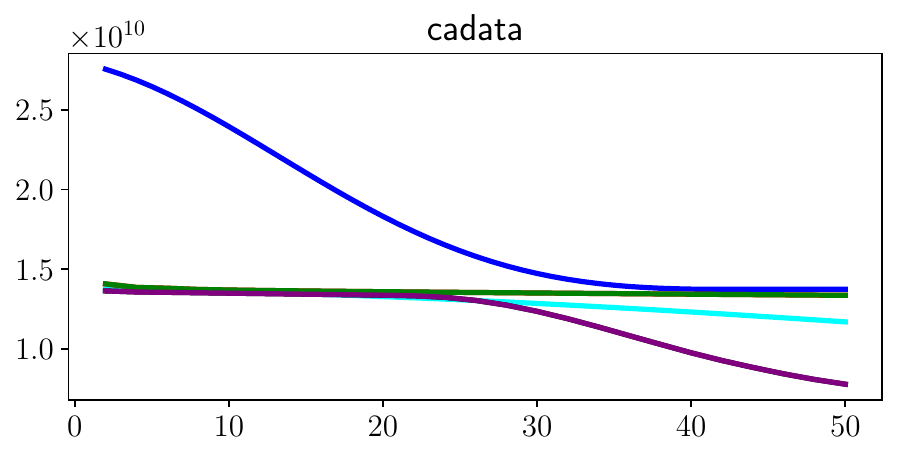}
\vspace{-3mm}
\caption{objective value}
\end{subfigure}

\begin{subfigure}[b]{\textwidth}
    \centering
    \includegraphics[width=0.3\textwidth]{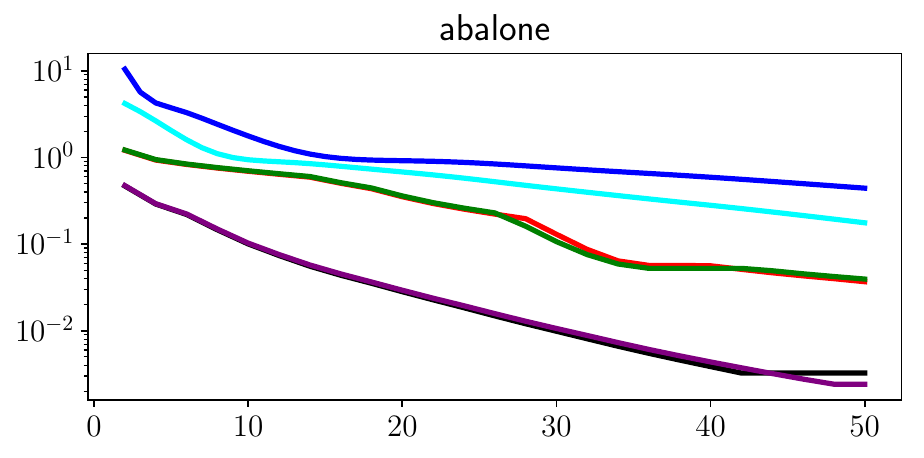}    
    \includegraphics[width=0.3\textwidth]{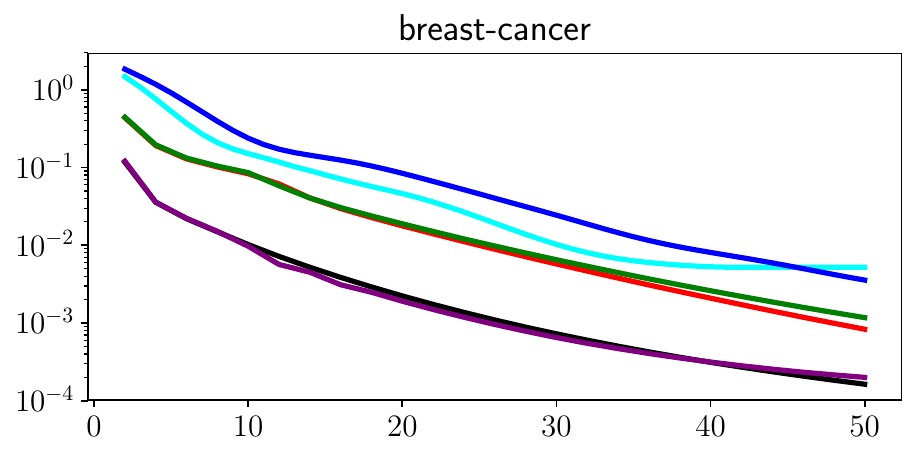}
    \includegraphics[width=0.3\textwidth]{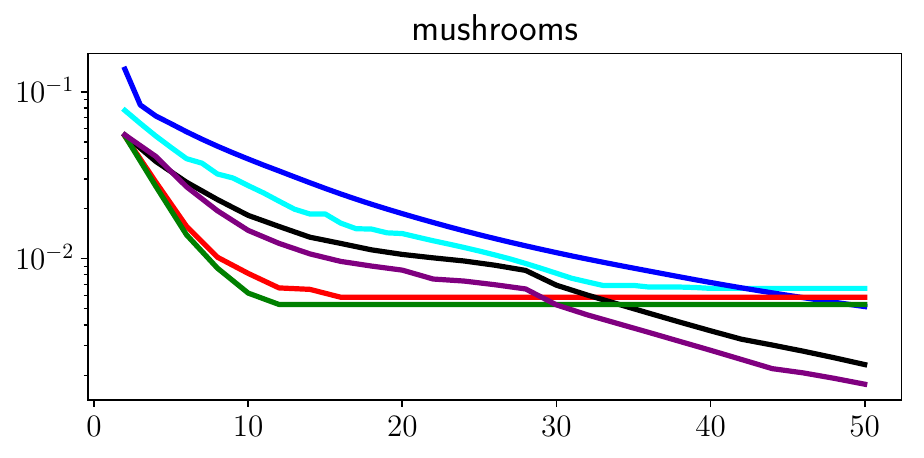}
    \\
    \includegraphics[width=0.3\textwidth]{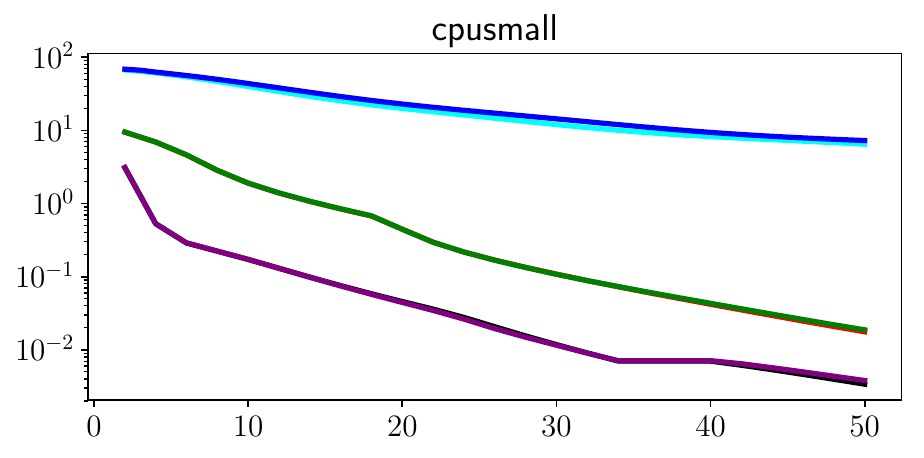}
    \includegraphics[width=0.3\textwidth]{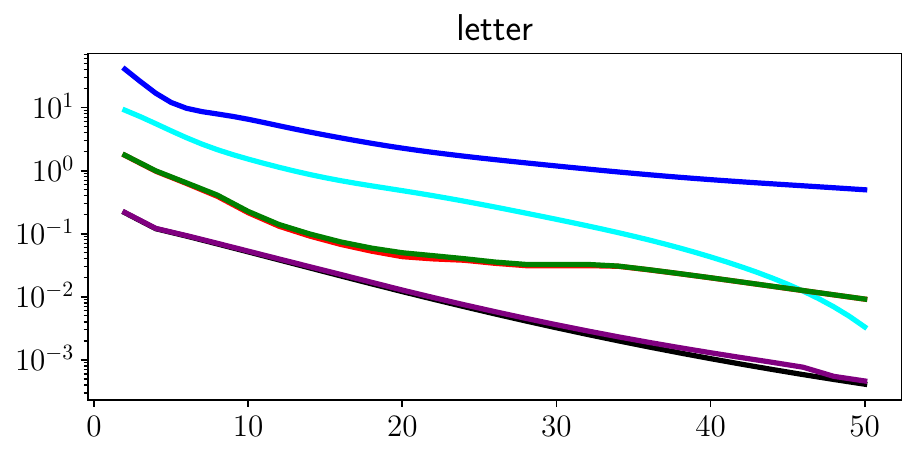}
    \includegraphics[width=0.3\textwidth]{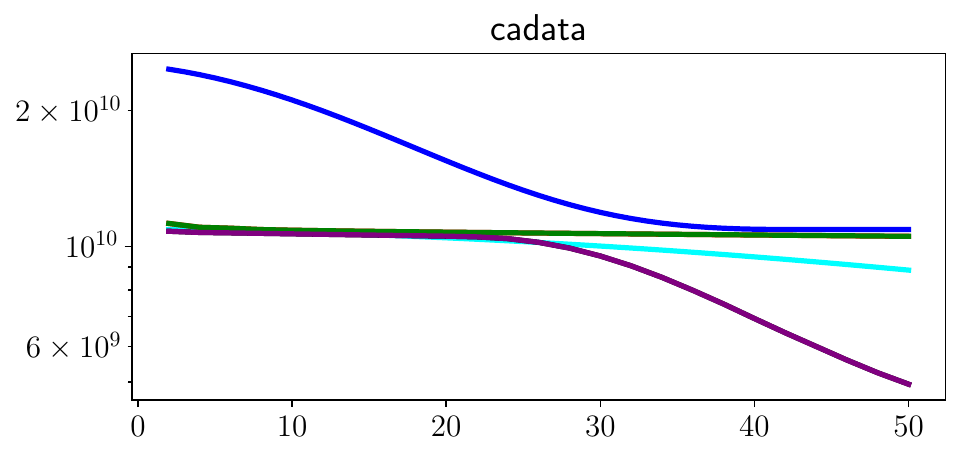}
\vspace{-3mm}
\caption{optimality gaps}
\end{subfigure}

\vspace{-5mm}
\caption{Comparison of inexact versions of FISTA, APG and ARMD for overlapping group Lasso problem. The vertical axis is the objective value (or optimality gap) and the horizontal axis is the number of computed gradients$/n$.}
\label{fig_resall_iasmd}
\end{figure*}

\section{Conclusion}
We have proposed a framework of accelerated randomized mirror descent algorithms for solving the large scale optimization problem \eqref{eq:composite} without the strongly convex assumption of $F^P(x)$. Our framework allows proximal points to be calculated inexactly and can achieve the optimal convergence rate. Using suitable parameters, our algorithms can obtain better complexity than APG. We also proposed a scheme for solving Problem \eqref{eq:composite} with non-smooth component functions $f_i$. Computational results affirm the effectiveness of our algorithms.

\appendix

\section*{Appendix: Technical Proofs} 

\noindent
\textbf{Proof of Lemma \ref{lem:variance}:}
%The following proof is also given in \cite[Corollary 3]{Xiao_Zhang}.
We have
\[
\begin{split}
\expect\norm{\grad F(y_{k,s}) - v_k}_*^2&=\expect\norm{1/(nq_{i_k})\bracket{\grad f_{i_k}(y_{k,s}) - \grad f_{i_k}(\tilde{x}_{s-1})}-(\grad F(y_{k,s})-\grad F(\tilde{x}_{s-1}))}_*^2\\
&\leq \expect \bracket{\norm{1/(nq_{i_k})\bracket{\grad f_{i_k}(y_{k,s}) - \grad f_{i_k}(\tilde{x}_{s-1})}}_* + \norm{\grad F(y_{k,s})-\grad F(\tilde{x}_{s-1})}_*}^2\\
&\leq 2\expect \frac{1}{(nq_{i_k})^2} \norm{\grad f_{i_k}(y_{k,s})-\grad f_{i_k}(\tilde{x}_{s-1})}_*^2 +2\norm{\grad F(y_{k,s})-\grad F(\tilde{x}_{s-1}) }_*^2.
\end{split}
\]
\noindent
\textbf{Proof of Lemma \ref{lem:bridge}:}
For notation succinctness, we omit the subscript $s$ when no confusing is caused. Applying Lemma \ref{lem:LiGrad}(1), we have:
\[
\begin{split}
F^P(x_k)&=\frac{1}{n}\sum\limits_{i=1}^n f_i(x_k) + P(x_k)\\
& \leq \frac{1}{n}\sum\limits_{i=1}^n\bracket{ f_i(y_k) + \iprod{\grad f_i(y_k)}{x_k-y_k}+\frac{L_i}{2}\norm{x_k-y_k}^2} + P(x_k)\\
&=F(y_k) + \iprod{\nabla F(y_k)-v_k}{x_k-y_k}+\frac{L_A}{2}\norm{x_k-y_k}^2 + P(x_k)+\iprod{v_k}{x_k-y_k}\\
&\leq F(y_k)+\frac{2L_Q}{\alpha_3} \norm{x_k-y_k}^2 + \frac{\alpha_3}{8 L_Q}\norm{\grad F(y_k) -v_k}_*^2+\frac{L_A}{2}\norm{x_k-y_k}^2 + P(x_k)+\iprod{v_k}{x_k-y_k},
\end{split}
\]
where the last inequality uses $\iprod{a}{b}\leq \frac{1}{2}\norm{a}_*^2 + \frac12\norm{b}^2$. Together with the update rule \eqref{eq:choosexk}, Lemma \ref{lem:DgeqE} with $\sigma=1$, and noting that $\hat{x}_k - y_k=\alpha_2(z_k-z_{k-1})$, we get:
\begin{align*}
\begin{split}
F^P(x_k)&\leq F(y_k) + \frac{\alpha_3}{8 L_Q}\norm{\grad F(y_k) -v_k}_*^2+ \iprod{v_k}{\hat{x}_k-y_k}+\frac{1}{2}
\bracket{L_A+\frac{4L_Q}{\alpha_3}}\norm{\hat{x}_k-y_k}^2+P(\hat{x}_k)\\
&=F(y_k) + \frac{\alpha_3}{8 L_Q}\norm{\grad F(y_k) -v_k}_*^2+ \alpha_2\iprod{v_k}{z_k-z_{k-1}}+\frac{1}{2}
\overline{L}\alpha_2^2 \norm{z_k-z_{k-1}}^2+P(\hat{x}_k)\\
&\leq F(y_k) + \frac{\alpha_3}{8 L_Q}\norm{\grad F(y_k) -v_k}^2+ \alpha_2\bracket{\iprod{v_k}{z_k-z_{k-1}}+\theta_s D(z_k,z_{k-1})} +P(\hat{x}_k).
\end{split}
\end{align*}

\noindent
\textbf{Proof of Lemma \ref{lem:approx}:}
Let $\phi_k(x)=\frac{1}{\theta_s }(\iprod{v_k}{x} + P(x))$, then  $\bar{z}_{k,s}=\argmin\limits_{x\in X_s} \{\phi_k(x)+D(x,z_{k-1,s})\}$. From Lemma \ref{lem:Dinequal}, for all $x\in X_s \cap \domain (P)$, we have:
\[
\frac{1}{\theta_s}(\iprod{v_k}{x} + P(x)) + D(x,z_{k-1,s}) \geq \min\limits_{x\in X_s} \{\phi_k(x) + D(x,z_{k-1,s})\} + D(x,\bar{z}_{k,s}).
\]
Together with $z_{k,s}\approx_{\varepsilon_{k,s}}\arg\min_{x\in X_s}\theta_s( \phi(x) + D(x,z_{k-1,s}))$, we get:
\begin{equation}\label{eq:Dzkzbar}
\iprod{v_k}{x} + P(x)+ \theta_sD(x,z_{k-1,s})\geq \iprod{v_k}{z_{k,s}}+P(z_{k,s}) + \theta_sD(z_{k,s},z_{k-1,s}) -\varepsilon_{k,s}+ \theta_s D(x,\bar{z}_{k,s}).
\end{equation}
From Lemma \ref{lem:triangle}, we get
$D(x,\bar{z}_{k,s})=D(x,z_{k,s})+ D(z_{k,s},\bar{z}_{k,s})-\iprod{x-z_{k,s}}{\grad h(\bar{z}_{k,s})-\grad h(z_{k,s})}.
$
Thus, the result follows.

\noindent
\textbf{Proof of Proposition \ref{prop:root}:}
For notation succinctness, we omit the subscript $s$ when no confusion is caused. Applying Lemma \ref{lem:bridge}, we have:
\begin{equation}\label{eq:i1}
F^P(x_k)\leq F(y_k) + \frac{\alpha_3}{8 L_Q}\norm{\grad F(y_k) -v_k}_*^2 + \alpha_2\bracket{\iprod{v_k}{z_k-z_{k-1}} + \theta_s D(z_k,z_{k-1})}
 + P(\hat{x}_k) .
\end{equation}
From Inequality \eqref{eq:i1} and Lemma \ref{lem:approx}, we deduce that:
\begin{equation}\label{eq:i2}
\begin{split}
F^P(x_k) &\leq F(y_k) + \frac{\alpha_3}{8 L_Q}\norm{\grad F(y_k) -v_k}^2+ \alpha_2( \iprod{v_k}{x-z_{k-1}}+P(x)-P(z_k)) + P(\hat{x}_k)\\
&+\alpha_2 \theta_s (D(x,z_{k-1})-D(x,z_k)-D(z_k,\bar{z}_k)+ \iprod{x-z_{k}}{\grad h(\bar{z}_{k})-\grad h(z_{k})} ) +\alpha_2\varepsilon_{k,s}.
\end{split}
\end{equation}
Taking expectation with respected to $i_k$ conditioned on $i_{k-1}$, and noting that $\expect_{i_k} [v_k]=\grad F(y_k)$ (for notation succinctness, we omit the subscript $i_k$ of the conditional expectation when it is clear in the context) and $P(\hat{x}_k) \leq \alpha_1 P(x_{k-1})+\alpha_2 P(z_k) + \alpha_3 P(\tilde{x}_{s-1})$,  it follows from \eqref{eq:i2} that:
\begin{equation}\label{eq:i3}
\begin{split}
&\expect F^P(x_k)\leq F(y_k) + \alpha_3\bracket{\frac{1}{8L_Q}\expect\norm{\grad F(y_k) - v_k}_*^2+ \iprod{\nabla F(y_k)}{\tilde{x}_{s-1} - y_k}}-\alpha_3 \iprod{\nabla F(y_k)}{\tilde{x}_{s-1} - y_k}\\
&+ \alpha_2\iprod{\grad F(y_k)}{x-z_{k-1}}+\alpha_2 P(x) +\alpha_1 P(x_{k-1})+\alpha_3 P(\tilde{x}_{s-1})+\alpha_2\theta_s ( D(x,z_{k-1}) - \expect D(x,z_k) ) + r_k,
\end{split}
\end{equation}
On the other hand, applying Lemma \ref{lem:variance}, the second inequality of Lemma \ref{lem:LiGrad} and noting that $\frac{1}{L_Q n q_i}\leq \frac{1}{L_i}$ and $\frac{1}{L_Q}\leq \frac{1}{L_A}$, we have:
\begin{align}\label{eq:var}
\begin{split}
&\frac{1}{8L_Q}\expect\norm{\grad F(y_k) - v_k}_*^2+  \iprod{\nabla F(y_k)}{\tilde{x}_{s-1} - y_k}\\
&\;\leq \frac{1}{4L_Q}\expect \frac{1}{(nq_{i_k})^2} \norm{\grad f_{i_k}(y_k)-\grad f_{i_k}(\tilde{x}_{s-1})}_*^2 + \frac{1}{4L_Q} \norm{\grad F(y_k)-\grad F(\tilde{x}_{s-1})}_*^2+ \iprod{\nabla F(y_k)}{\tilde{x}_{s-1} - y_k}\\
&\; =\frac{1}{n}\sum\limits_{i=1}^n \frac{1}{4L_Q} \frac{1}{nq_i}\norm{\grad f_i(y_k)-\grad f_i(\tilde{x}_{s-1})}_*^2 + \frac{1}{2n}\sum\limits_{i=1}^n \iprod{\grad f_i(y_k)}{\tilde{x}_{s-1} - y_k}+\frac{1}{2 }\iprod{\nabla F(y_k)}{\tilde{x}_{s-1} - y_k}\\&\quad\qquad +\frac{1}{4L_Q} \norm{\grad F(y_k)-\grad F(\tilde{x}_{s-1})}_*^2  \\
&\; \leq \frac{1}{2n}\sum\limits_{i=1}^n \bracket{\frac{1}{2L_i} \norm{\grad f_i(y_k)-\grad f_i(\tilde{x}_{s-1})}_*^2  +\iprod{\grad f_i(y_k)}{\tilde{x}_{s-1} - y_k}}\\
&\qquad\qquad+\frac12\bracket{\frac{1}{2L_A}\norm{\grad F(y_k)-\grad F(\tilde{x}_{s-1})}_*^2  + \iprod{\nabla F(y_k)}{\tilde{x}_{s-1} - y_k} }\\
&\;\leq \frac{1}{2n}\sum\limits_{i=1}^n (f_i(\tilde{x}_{s-1})-f_i(y_k))+\frac{1}{2}\bracket{F(\tilde{x}_{s-1}) - F(y_k)}= F(\tilde{x}_{s-1}) - F(y_k).
\end{split}
\end{align}
Therefore, \eqref{eq:i3} and \eqref{eq:var} imply that:
\[\begin{split}
\expect F^P(x_k)& \leq (1-\alpha_3)F(y_k) + \alpha_3 F^P(\tilde{x}_{s-1}) + \alpha_2\iprod{\grad F(y_k)}{x-y_k} +\alpha_2 P(x)+ \alpha_2\iprod{\grad F(y_k)}{y_k-z_{k-1}} \\&\qquad- \alpha_3 \iprod{\grad F(y_k)}{\tilde{x}_{s-1} -y_k}+\alpha_1 P(x_{k-1}) + \alpha_2 \theta_s (D(x,z_{k-1}) - \expect D(x,z_k))+r_k\\
&\lea (1-\alpha_3)F(y_k) + \alpha_3 F^P(\tilde{x}_{s-1})+\alpha_2(F(x)-F(y_k))+\alpha_2 P(x)\\
&\qquad + \alpha_1\iprod{\grad F(y_k)}{x_{k-1}-y_k} + \alpha_1 P(x_{k-1})+ \alpha_2 \theta_s (D(x,z_{k-1}) - \expect D(x,z_k))+r_k\\
&\leb (1-\alpha_3-\alpha_2) F(y_k) +  \alpha_3 F^P(\tilde{x}_{s-1})+ \alpha_2 F^P(x) \\
&\qquad+ \alpha_1(F(x_{k-1}) - F(y_k)) + \alpha_1 P(x_{k-1})+ \alpha_2 \theta_s (D(x,z_{k-1}) - \expect D(x,z_k))+r_k\\
& = \alpha_1 F^P(x_{k-1}) + \alpha_2 F^P(x) + \alpha_3 F^P(\tilde{x}_{s-1})+ \alpha_2 \theta_s (D(x,z_{k-1}) - \expect D(x,z_k))+r_k.
\end{split}
\]
Here in (a) we use $\iprod{\grad F(y_k)}{x-y_k} \leq F(x)-F(y_k)$ and 
$\alpha_2(y_k-z_{k-1})-\alpha_3(\tilde{x}_{s-1}-y_k)=\alpha_1(x_{k-1}-y_k)$, in (b) we use $\iprod{\grad F(y_k)}{x_{k-1}-y_k} \leq F(x_{k-1})-F(y_k)$. Finally, we take expectation with respected to $i_{k-1}$ to get the result.

\noindent
\textbf{Proof of Proposition \ref{prop:result}:}
Applying Proposition \ref{prop:root} with $x=x^*$ we have: 
\[\begin{split}
\expect (F^P(x_{k,s}) -F^P(x^*))&\leq \alpha_{1,s} \expect (F^P(x_{k-1,s})-F^P(x^*))+\alpha_{3}(F^P(\tilde{x}_{s-1})-F^P(x^*)) \\
&\qquad + \alpha_{2,s}^2\overline{L} (\expect D(x^*,z_{k-1,s}) - \expect D(x^*,z_{k,s}))+ r^*_{k,s}.
\end{split}
\]
Denote $d_{k,s}=\expect (F^P(x_{k,s}) -F^P(x^*))$, then
\[d_{k,s} \leq \alpha_{1,s} d_{k-1,s} + \alpha_{3} \tilde{d}_{s-1} + \alpha_{2,s}^2\overline{L} (\expect D(x^*,z_{k-1,s}) - \expect D(x^*,z_{k,s}))+ r^*_{k,s},
\]
which implies 
$
\frac{1}{\alpha_{2,s}^2}d_{k,s} \leq \frac{\alpha_{1,s}}{\alpha_{2,s}^2} d_{k-1,s} + \frac{\alpha_{3}}{\alpha_{2,s}^2}\tilde{d}_{s-1} + \overline{L}  (\expect D(x^*,z_{k-1,s}) - \expect D(x^*,z_{k,s}))+\frac{ r^*_{k,s}}{\alpha_{2,s}^2}.
$
Summing up this inequality from $k=1$ to $k=m$ we get:
\[
\begin{split}
\frac{1}{\alpha_{2,s}^2}d_{m,s}+\frac{1-\alpha_{1,s}}{\alpha_{2,s}^2}\sum\limits_{k=1}^{m-1} d_{k,s}\leq\frac{\alpha_{1,s}}{\alpha_{2,s}^2}d_{0,s}+ \frac{\alpha_{3}}{\alpha_{2,s}^2}m \tilde{d}_{s-1} + \overline{L} \bracket{\expect D(x^*,z_{0,s}) - \expect D(x^*,z_{m,s}) } +\frac{\sum_{k=1}^m   r^*_{k,s}}{\alpha_{2,s}^2}.
\end{split}
\]
Using the update rule \eqref{eq:xsupdate}, $\alpha_{1,s}+\alpha_{3}=1-\alpha_{2,s}$, $z_{m,s-1}=z_{0,s}$ and $d_{m,s-1}=d_{0,s}$ we get:
\[
\begin{split}
\frac{1}{\alpha_{2,s}^2}d_{m,s}+\frac{1-\alpha_{1,s}}{\alpha_{2,s}^2}\sum\limits_{k=1}^{m-1} d_{k,s} &\leq \frac{1-\alpha_{2,s}}{\alpha_{2,s}^2}d_{m,s-1} + \frac{\alpha_{3}}{\alpha_{2,s}^2} \sum\limits_{k=1}^{m-1} d_{k,s-1} \\
&\quad + \overline{L} \bracket{\expect D(x^*,z_{m,s-1}) - \expect D(x^*,z_{m,s})}+\frac{\sum_{k=1}^m  r^*_{k,s}}{\alpha_{2,s}^2}.
\end{split}
\]
Combining with the update rule \eqref{eq:alphaupdate} we obtain:
\begin{equation}\label{eq:recursive}
\begin{split}
\frac{1-\alpha_{2,s+1}}{\alpha_{2,s+1}^2}d_{m,s}+\frac{\alpha_{3}}{\alpha_{2,s+1}^2} \sum\limits_{k=1}^{m-1} d_{k,s}
&\leq \frac{1-\alpha_{2,s}}{\alpha_{2,s}^2} d_{m,s-1} + \frac{\alpha_{3}}{\alpha_{2,s}^2} \sum\limits_{k=1}^{m-1} d_{k,s-1} \\
&\quad + \overline{L} \bracket{\expect D(x^*,z_{m,s-1}) - \expect D(x^*,z_{m,s})}+\frac{\sum_{k=1}^m  r^*_{k,s}}{\alpha_{2,s}^2}.
\end{split}
\end{equation}
Therefore, 
\[\begin{split}
\frac{\alpha_{3}}{\alpha_{2,s+1}^2} m\tilde{d}_s &\lea
\frac{\alpha_{3}}{\alpha_{2,s+1}^2}\sum\limits_{k=1}^{m} d_{k,s}
\leb \frac{1-\alpha_{2,s+1}}{\alpha_{2,s+1}^2}d_{m,s}+\frac{\alpha_{3}}{\alpha_{2,s+1}^2} \sum\limits_{k=1}^{m-1} d_{k,s}\\
&\lec \frac{1-\alpha_{2,1}}{\alpha_{2,1}^2} d_{m,0}+\frac{\alpha_{3}}{\alpha_{2,1}^2}\sum\limits_{k=1}^{m-1}d_{k,0}+\overline{L}  \bracket{ \expect D(x^*,z_{m,0}) - \expect D(x^*,z_{m,s})}+\sum\limits_{i=1}^s\frac{\sum_{k=1}^m r^*_{k,i}}{\alpha_{2,i}^2},
\end{split}
\]
where in (a) we use the update rule \eqref{eq:xsupdate}, in (b) we use the property $\alpha_3 \leq 1-\alpha_{2,s+1}$, and in (c) we use the recursive inequality \eqref{eq:recursive}. The result then follows.

\noindent
\textbf{Proof of Theorem \ref{thrm:main}:}
Without loss of generality, we can assume that:
$$
\frac{1}{m}\bracket{\frac{4(1-\alpha_{2,1})}{\alpha_{2,1}^2\alpha_3} d_{m,0}+\frac{4}{\alpha_{2,1}^2}\sum\limits_{i=1}^{m-1}d_{i,0}}=O(F^P(\tilde{x}_0)-F^P(x^*)).
$$
When $\varepsilon_{k,s}=0$, then $z_{k,s}=\bar{z}_{k,s}$ and we have $r_{k,s}=0$. The convergence rate of exact ASMD follows from Proposition \ref{prop:result} by taking $\alpha_{2,s}=\frac{2}{s+2}$ and noting that $D(x^*,z_{m,s})\geq 0$.

\noindent
\textbf{Proof of Theorem \ref{thrm:main_inexact}:}
We remind that Inequality \eqref{eq:Dzkzbar} holds for all $x$. Taking $x=z_{k,s}$, \eqref{eq:Dzkzbar} yields that $D(z_{k,s},\bar{z}_{k,s})\leq \frac{\varepsilon_{k,s}}{\theta_s}$. On the other hand, if $ h(\cdot)$ is $L_h$-Lipschitz smooth, then:
\[\norm{\grad h(\bar{z}_{k,s})-\grad h(z_{k,s})}\leq L_h \norm{\bar{z}_{k,s}-z_{k,s}}\leq L_h\sqrt{2D(z_{k,s},\bar{z}_{k,s})}\leq L_h\sqrt{\frac{2\varepsilon_{k,s}}{\theta_s}}.
\]
If $\norm{z_{k,s}}\leq  C$ then we let $C_1=\norm{x^*}+C$. Noting that $D(z_{k,s},\bar{z}_{k,s})\geq 0$, we have 
\[
\begin{split}
{r}^*_{k,s} \leq \alpha_{2,s}\theta_s\norm{x^*-z_{k,s}}\norm{\grad h(\bar{z}_{k,s})-\grad h(z_{k,s})}+ \alpha_{2,s}\varepsilon_{k,s} \leq \alpha_{2,s} C_1 L_h\sqrt{2\epsilon_s\theta_s}+ \alpha_{2,s}\epsilon_s. 
\end{split}
\]
Hence, 
\begin{align}
\label{temp1}
\alpha_{2,s+1}^2\sum\limits_{i=1}^s\sum\limits_{k=1}^m \frac{{r}^*_{k,i}}{m\alpha_3\alpha_{2,i}^2} \leq \alpha_{2,s+1}^2 \sum\limits_{i=1}^s\bracket{ \frac{C_1L_h\sqrt{2\epsilon_i\bar{L}}}{\alpha_3\sqrt{\alpha_{2,i}}}+\frac{\epsilon_i}{\alpha_3\alpha_{2,i}}}.
\end{align} 
If the adaptive inexact rule $\max\left\{\norm{\bar{z}_{k,s}}^2\varepsilon_{k,s},C\varepsilon_{k,s}\right\}\leq C\epsilon_s$ is chosen, we have
\[
\begin{split}
{r}^*_{k,s} &\leq \alpha_{2,s}\theta_s\bracket{\norm{x^*}+\norm{\bar{z}_{k,s}}+\norm{\bar{z}_{k,s}-z_{k,s}}}\norm{\grad h(\bar{z}_{k,s})-\grad h(z_{k,s})}+ \alpha_{2,s}\varepsilon_{k,s}\\
& \leq \alpha_{2,s} \norm{x^*} L_h\sqrt{2\epsilon_s\theta_s}+ \alpha_{2,s} L_h  \sqrt{2C\epsilon_s\theta_s} + \alpha_{2,s}L_h 2 \epsilon_s+ \alpha_{2,s}\epsilon_{s}. 
\end{split}
\]
In this case, we let $C_1=\norm{x^*}+\sqrt{C}$. We then have 
\begin{align}
\label{temp2}
\alpha_{2,s+1}^2\sum\limits_{i=1}^s\sum\limits_{k=1}^m \frac{{r}^*_{k,i}}{m\alpha_3\alpha_{2,i}^2} \leq \alpha_{2,s+1}^2 \sum\limits_{i=1}^s\bracket{ \frac{C_1L_h\sqrt{2\epsilon_i\bar{L}}}{\alpha_3\sqrt{\alpha_{2,i}}}+\frac{(2L_h+1)\epsilon_i}{\alpha_3\alpha_{2,i}}}.
\end{align} 
The result then follows from \eqref{temp1}, \eqref{temp2} and Proposition \ref{prop:result} easily.

\noindent
\textbf{Proof of Theorem \ref{thrm:smoothcase}:}
Let $x^*_\mu$ is the optimal solution of Problem \eqref{eq:smoothprob}. We have:
\begin{equation}\label{eq:smoothproof}
\expect F^P_\mu(\tilde{x}_{\mu,s})-F^P_\mu(x^*_\mu)=O\bracket{\frac{1+\frac{4\overline{L}_\mu}{m}+\bar{C}}{(s+3)^2}},
\end{equation}
where $\bar{C}=O\bracket{\sqrt{\bar{L}_\mu}}$, by applying Theorem \ref{thrm:main_inexact}. 
By Assumption \ref{asstn:smooth}, we have:
\[
\begin{split}
\expect F^P(\tilde{x}_{\mu,s})-F^P(x^*)& = \expect F(\tilde{x}_{\mu,s}) + \expect P (\tilde{x}_{\mu,s}) - F(x^*) - P(x^*)\\
& \leq \expect F_\mu(\tilde{x}_{\mu,s}) +  \overline{K} \mu + \expect P (\tilde{x}_{\mu,s}) - F(x^*)-P(x^*) \\
&\leq\expect F_\mu(\tilde{x}_{\mu,s})+  \overline{K} \mu + \expect P (\tilde{x}_{\mu,s})  -F_\mu(x^*) + \underline{K}\mu-P(x^*) \\
& \leq \expect F^P_\mu(\tilde{x}_{\mu,s})-F^P_\mu(x^*)+ \bracket{\overline{K}+\underline{K}} \mu.
\end{split}
\]
Together with \eqref{eq:smoothproof} and noting that $F^P_\mu(x^*)\geq F^P_\mu(x^*_\mu)$, we get:
\[
\begin{split}
\expect F^P(\tilde{x}_{\mu,s})-F^P(x^*)\leq  \expect F^P_\mu(\tilde{x}_{\mu,s})-F^P_\mu(x_\mu^*)+ \bracket{\overline{K}+\underline{K}} \mu=O\bracket{\frac{1+\frac{4\overline{L}_\mu}{m}+\bar{C}}{(s+3)^2}} + \bracket{\overline{K}+\underline{K}} \mu.
\end{split}
\]

\bibliographystyle{abbrv}
\scriptsize{
\bibliography{ASMD}
}

\end{document}